\documentclass[oneside,english]{amsart}
\usepackage{lmodern}
\usepackage[latin9]{inputenc}
\usepackage{textcomp}
\usepackage{amsthm}
\usepackage{amstext}
\usepackage{amssymb}
\usepackage{esint}

\makeatletter

\newcommand{\lyxmathsym}[1]{\ifmmode\begingroup\def\b@ld{bold}
  \text{\ifx\math@version\b@ld\bfseries\fi#1}\endgroup\else#1\fi}

\numberwithin{equation}{section}
\numberwithin{figure}{section}
\theoremstyle{plain}
\newtheorem{thm}{\protect\theoremname}
  \theoremstyle{plain}
  \newtheorem{prop}[thm]{\protect\propositionname}
  \theoremstyle{plain}
  \newtheorem{cor}[thm]{\protect\corollaryname}
  \theoremstyle{plain}
  \newtheorem{lem}[thm]{\protect\lemmaname}

\numberwithin{thm}{section}

\makeatother

\usepackage{babel}
  \providecommand{\corollaryname}{Corollary}
  \providecommand{\lemmaname}{Lemma}
  \providecommand{\propositionname}{Proposition}
\providecommand{\theoremname}{Theorem}

\begin{document}

\title{Intersection Theory and the Horn Inequalities \\for Invariant Subspaces}

\thanks{The authors were supported in part by grants from the National Science
Foundation.}

\author{H. Bercovici and W. S. Li}

\address{Departent of Mathematics, Indiana University, Bloomington, IN 47405}

\email{bercovic@indiana.edu}

\address{School of Mathematics, Georgia Institute of Technology, Atlanta,
GA 30332}

\email{li@math.gatech.edu}
\begin{abstract}
We provide a direct, intersection theoretic, argument that the Jordan
models of an operator of class $C_{0}$, of its restriction to an
invariant subspace, and of its compression to the orthogonal complement,
satisfy a multiplicative form of the Horn inequalities, where `inequality'
is replaced by `divisibility'. When one of these inequalities is saturated,
we show that there exists a splitting of the operator into quasidirect
summands which induces similar splittings for the restriction of the
operator to the given invariant subspace and its compression to the
orthogonal complement. The result is true even for operators acting
on nonseparable Hilbert spaces. For such operators the usual Horn
inequalities are supplemented so as to apply to all the Jordan blocks
in the model.
\end{abstract}
\maketitle

\section{Introduction\label{sec:Introduction}}

\markboth{}{}Consider a complex Hilbert space $\mathcal{H}$ and
an operator $T$ of class $C_{0}$ acting on it. It is known (see
\cite{C0-book,SN}) that $T$ is quasisimilar to a uniquely determined
Jordan model, that is, to an operator of the form
\[
\bigoplus_{1\le n<\aleph}S(\theta_{n}),
\]
where the sum is indexed by ordinal numbers $n$ less than some cardinal
$\aleph$ and each $\theta_{n}$ is an inner function in the unit
disk such that $\theta_{n}$ divides $\theta_{m}$ if ${\rm card}(m)\le{\rm card}(n)$.
If $\mathcal{H}'$ is an invariant subspace for $T$, the restriction
$T|\mathcal{H}'$ and the compression $P_{\mathcal{H}''}T|\mathcal{H}''$,
$\mathcal{H}''=\mathcal{H}\ominus\mathcal{H}'$, are of class $C_{0}$
as well, and therefore they have Jordan models, say
\[
\bigoplus_{1\le n<\aleph}S(\theta_{n}')\text{ and }\bigoplus_{1\le n<\aleph}S(\theta_{n}^{\prime\prime}).
\]
(Note that the first ordinal in the sum is $1$ rather than $0$ as
in \cite{C0-book}. This way of labeling the Jordan blocks is more
convenient for stating the Horn inequalities.) It has been known for
some time \cite{BLT-on-LR1,BLT-on-LR2,li-mull1} that the functions
$\{\theta_{n},\theta'_{n},\theta_{n}^{\prime\prime}:1\le n<\aleph_{0}\}$
satisfy a version of the Littlewood-Richardson rule provided that
\[
\bigwedge_{n=1}^{\infty}\theta_{n}\equiv1
\]
 and, due to results of \cite{kly,KT}, this is equivalent (in the
case of finite multiplicity $N$) to saying that $\prod_{n=1}^{N}\theta_{n}=\prod_{n=1}^{N}(\theta'_{n}\theta_{n}^{\prime\prime})$,
and that these functions satisfy a collection of divisibility relations,
analogous to the Horn inequalities. The collection of these divisibility
relations is indexed by triples of Schubert cells in a Grassmann variety
whose intersection has dimension zero. It is natural to ask whether
a direct connection between these divisibility relations and intersection
theory can be made. Indeed, in the context of finitely generated torsion
modules over a principal ideal domain, such a connection was made
in \cite{BDL-Szeged}, and it is our purpose to extend that approach
to the context of arbitrary $C_{0}$ operators. The result is that
a sufficient number of these divisibility relations can be obtained
from special invariant subspaces $\mathcal{M}$ of $T$. More precisely,
assume that $T|\mathcal{M}$ has cyclic multiplicity $r<\infty$,
set $\mathcal{M}'=\mathcal{M}\cap\mathcal{H}'$, $\mathcal{M}''=\overline{P_{\mathcal{H}''}\mathcal{M}}$,
and let
\[
\bigoplus_{n=1}^{r}S(\alpha_{n}),\quad\bigoplus_{n=1}^{r}S(\alpha'_{n}),\quad\bigoplus_{n=1}^{r}S(\alpha_{n}^{\prime\prime})
\]
be the Jordan models of $T|\mathcal{M}$, $T|\mathcal{M}'$, and $P_{\mathcal{M}''}T|\mathcal{M}''$,
respectively. It is known \cite{C0-book} that 
\[
\prod_{n=1}^{r}\alpha_{n}\equiv\prod_{n=1}^{r}(\alpha_{n}'\alpha_{n}^{\prime\prime}),
\]
where we use the notation $\varphi\equiv\psi$ to indicate that the
quotient of the inner functions $\varphi$ and $\psi$ is a constant
(necessarily of absolute value equal to $1$). We show that $\mathcal{M}$
can be chosen in such a way that $\prod_{n=1}^{r}\alpha_{n}$ is divisible
by a product of the form $\prod_{n=1}^{r}\theta_{i_{n}}$, while the
products $\prod_{n=1}^{r}\alpha'_{n},\prod_{n=1}^{r}\alpha_{n}^{\prime\prime}$
divide $\prod_{n=1}^{r}\theta'_{j_{n}},\prod_{n=1}^{r}\theta_{k_{n}}^{\prime\prime}$
respectively, thus establishing that $\prod_{n=1}^{r}\theta_{i_{n}}$
divides $\prod_{n=1}^{r}(\theta'_{j_{n}}\theta_{k_{n}}^{\prime\prime})$
for certain indices $\{k_{i},\ell_{i},m_{i}:i=1,2,\dots,r\}$. Moreover,
in case this last divisibility is an equality, we show that $\mathcal{M}$
is, in a weak sense, a reducing subspace for $T$, with similar statements
about $\mathcal{M}'$ and $\mathcal{M}''$. Two special cases of this
result were proved in \cite{extremal-structure}, but the methods
of that paper do not seem to extend beyond the two classes of inequalities
considered there. The existence of such (almost) reducing spaces is
analogous to the existence of common reducing spaces for selfadjoint
matrices whose sum saturates one of the Horn inequalities (see \cite{fulton-survey}).

When the space $\mathcal{H}$ is separable, the summands $S(\theta_{n})$
in the model of $T$ are constant for $n\ge\aleph_{0}$, which means
that $S(\theta_{n})$ acts on a space of dimension zero which can
then be omitted from the sum. One may ask what form the divisibility
relations take in the nonseparable case. The only additional relations
state that $\theta_{n}$ divides $\theta_{n}'\theta_{n}^{\prime\prime}$
when $n\ge\aleph_{0}$. Just as in the case of the Horn relations,
these divisibility relations can be obtained by exhibiting an invariant
subspace $\mathcal{M}$ for $T$ which, in this case, is reducing
in the usual sense for $T$ and for $P_{\mathcal{H}'}$.

The divisibility relations we consider were studied earlier when $\mathcal{H}$
is finite dimensional. We refor to \cite{thompson} for a survey of
these results. The Littlewood-Richardson rule in this context and,
indeed, in the case of finitely generated modules over a discrete
valuation ring, was first proved in \cite{klein-1} (see also \cite{macD}
for a different argument.) The basic ideas in this paper originated
in the study of singular numbers for products of operators \cite{sing-val}
and in the study of torsion modules over principal ideal domains \cite{BDL-Szeged}.
The techniques we use are necessarily different, and they may obscure
to some extent the essential simplicity of the arguments.

The remainder of this paper is organized as follows. Section \ref{sec:Preliminaries}
contains some preliminaries about the class $C_{0}$ and intersection
theory. In Section \ref{sec:Invariant-subspaces-vs-Schubert} we consider
special invariant subspaces for operators of class $C_{0}$ with finite
defect indices. In Section \ref{sec:The-Horn-inequalities} we prove
the Horn divisibility relations, first for contractions with finite
multiplicity and then in general. The relations pertaining to $\theta_{n}$
for $n\ge\aleph_{0}$ are established in Section \ref{sec:Operators-on-nonseparable}.
We conclude in Section \ref{sec:inverse problem} with a discussion
of the `inverse' problem: given Jordan operators
\[
J=\bigoplus_{1\le n<\aleph}S(\theta_{n}),\quad J'=\bigoplus_{1\le n<\aleph}S(\theta'_{n}),\quad J''=\bigoplus_{1\le n<\aleph}S(\theta_{n}^{\prime\prime}),
\]
do there exist a $C_{0}$ operator $T$ and an invariant subspace
$\mathcal{H}'$ for $T$ such that $T$, $T|\mathcal{H}'$, and $P_{\mathcal{H}^{\prime\perp}}T|\mathcal{H}^{\prime\perp}$
are quasisimilar to $J,$ $J'$, and $J''$, respectively?

\section{Preliminaries\label{sec:Preliminaries}}

Recall \cite{SN} that an operator $T$ acting on a complex Hilbert
space $\mathcal{H}$ is a \emph{contraction} if $\|T\|\le1$, and
it is a \emph{completely nonunitary contraction} if, in addition,
$T$ has no unitary restriction to any invariant subspace. Given a
completely nonunitary contraction $T$ on $\mathcal{H}$, the usual
polynomial calculus $p\mapsto p(T)$ extends to a functional calculus
(discovered by Sz.-Nagy and Foias) defined on the algebra $H^{\infty}$
of bounded analytic functions  in the unit disk $\mathbb{D}=\{\lambda\in\mathbb{C}:|\lambda|<1\}$.
The operator $T$ is of \emph{class} $C_{0}$ if the ideal $\mathcal{J}_{T}=\{u\in H^{\infty}:u(T)=0\}$
is not zero, in which case $\mathcal{J}_{T}$ is a principal ideal
generated by an inner function $m_{T}$, uniquely determined up to
a scalar factor and called the \emph{minimal function} of $T$.

The simplest operators of class $C_{0}$ are the \emph{Jordan blocks}.
Given an inner function $\theta\in H^{\infty}$, the Jordan block
$S(\theta)$ is obtained by compressing the unilateral shift $S$
on the Hardy space $H^{2}$ to its co-invariant subspace
\[
\mathcal{H}(\theta)=H^{2}\ominus\theta H^{2}.
\]
A \emph{Jordan operator} is, as already indicated in the introduction,
an operator of the form
\[
J=\bigoplus_{1\le n<\aleph}S(\theta_{n}),
\]
where $\aleph$ is a cardinal number (that is, the smallest ordinal
of some cardinality) and, for each ordinal $n$, $\theta_{n}$ is
an inner function such that $\theta_{n}|\theta_{m}$ whenever ${\rm card}(n)\ge{\rm card}(m)$.
(We write $\varphi|\psi$ to indicate that $\varphi$ is a divisor
of $\psi$ in $H^{\infty}$.) In particular, every $\theta_{n}$ divides
$\theta_{1}$ which is in fact the minimal function of $J$. Note
that, when $\theta$ is a constant inner function, we have $\mathcal{H}(\theta)=\{0\}$,
so summands with $\theta_{n}\equiv1$ do not contribute to the sum
defining a Jordan operator. One is tempted to write $J$ as a direct
sum extended over the entire class of ordinal numbers (with $\theta_{n}=1$
for sufficiently large $n$), but this does not seem wise if set theoretical
decorum is to be maintained.

The class $C_{0}$ is completely classified by the relation of quasisimilarity.
Two operators $T,T'$ acting on $\mathcal{H},\mathcal{H}'$, respectively,
are said to be \emph{quasisimilar} if there exist continuous linear
operators $X:\mathcal{H}\to\mathcal{H}'$ and $Y:\mathcal{H}'\to\mathcal{H}$
which are one-to-one, have dense ranges, and satisfy the following
intertwining relations
\[
XT=T'X,\quad TY=YT'.
\]
We write $T\sim T'$ to indicate that $T$ is quasisimilar to $T'$.
Quasisimliarity is a weaker relation than similarity, but it is just
right for the class $C_{0}$. The following result is \cite[Theorem III.5.23]{C0-book}.
\begin{thm}
\label{thm:Jordan-exists}The quasisimilarity equivalence class of
every operator of class $C_{0}$ contains a unique Jordan operator,
called the Jordan model of $T$.
\end{thm}
Given an operator $T$ of class $C_{0}$, we may occasionally write
$\theta_{n}^{T}$ for the inner functions in its Jordan model. In
order to characterize these functions, we need the concept of \emph{cyclic
multiplicity} for an operator $T$. This is a cardinal number $\mu_{T}$
defined as the smallest cardinality of a subset $C\subset\mathcal{H}$
with the property that the invariant subspace for $T$ generated by
$C$ is the entire space $\mathcal{H}$:
\[
\mathcal{H}=\bigvee\{T^{k}h:h\in C,k\ge0\}.
\]

\begin{prop}
\label{prop:theta-from-multip}\emph{(\cite[Corollary III.3.25]{C0-book})}Given
an operator $T$ of class $C_{0}$ on $\mathcal{H}$, an inner function
$\varphi$, and an ordinal number $n\ge1$, the following statements
are equivalent:
\begin{enumerate}
\item $\theta_{n+1}^{T}|\varphi$.
\item The cyclic multiplicity of the restriction $T|\overline{\varphi(T)\mathcal{H}}$
is at most ${\rm card}(n)$.
\end{enumerate}
\end{prop}
Note again the shift $n\mapsto n+1$ compared to the corresponding
statement in \cite{C0-book}. This is immaterial for transfinite $n$
for which we have $\theta_{n+1}^{T}\equiv\theta_{n}^{T}$ as well
as ${\rm card}(n+1)={\rm card}(n)$. One way to use this result is
as follows. If $\theta_{n}^{T}\not\equiv1$ for some finite $n$,
then $T$ has a restriction to some invariant subspace which is quasisimilar
to the direct sum of $n$ copies of $S(\varphi)$ for any nonconstant
inner divisor $\varphi$ of $\theta_{n}^{T}$. This follows from the
fact that $S(\varphi)$ is unitarily equivalent to the restriction
of $S(\psi)$ to some invariant subspace provided that $\varphi|\psi$.
The (unique) invariant subspace in question is $(\psi/\varphi)H^{2}\ominus\psi H^{2}$.
If $\theta_{n}^{T}\not\equiv1$ for some transfinite $n$, then $T$
has a restriction which is quasisimilar to the direct sum of $m$
copies of $S(\varphi)$, $\varphi|\theta_{n}^{T}$, where $m$ is
the smallest cardinal \emph{greater than $n$. }Indeed $m$ is precisely
the cardinality of the set of ordinals $\{n':{\rm card}(n')={\rm card}(n)\}$.

Another way to describe the structure of an operator in terms of its
Jordan model is to use \emph{quasidirect} decompositions of the Hilbert
space. Let $\{\mathcal{H}_{i}\}_{i\in I}$ be a collection of closed
subspaces of the Hilbert space $\mathcal{H}$. We say that $\mathcal{H}$
is the \emph{quasidirect sum} of the spaces $\{\mathcal{H}_{i}\}_{i\in I}$
if 
\[
\mathcal{H}=\bigvee_{i\in I}\mathcal{H}_{i}
\]
 and, for each $j\in I$, we have
\[
\mathcal{H}_{j}\cap\left[\bigvee_{i\ne j}\mathcal{H}_{i}\right]=\{0\}.
\]
Given a quasidirect decomposition $\mathcal{H}=\mathcal{M}\vee\mathcal{N}$,
we say that $\mathcal{N}$ is a \emph{quasidirect complement} of $\mathcal{M}$.
\begin{thm}
\label{thm:quasi-dir}\emph{(\cite[Theorem III.6.10]{C0-book})}Let
$T$ be an operator of class $C_{0}$ on $\mathcal{H}$. For each
ordinal $n\ge1$, there exists an invariant subspace $\mathcal{H}_{n}\subset\mathcal{H}$
such that $T|\mathcal{H}_{n}$ is quasisimilar to $S(\theta_{n}^{T})$,
and $\mathcal{H}$ is the quasidirect sum of the spaces $\{\mathcal{H}_{n}\}$.
In addition, $\mathcal{H}_{n+i}\perp\mathcal{H}_{m+j}$ if $n$ and
$m$ are distinct limit ordinals and $i,j<\aleph_{0}$.
\end{thm}
We note again that we do not run into any set theory difficulties
because $\mathcal{H}_{n}=\{0\}$ for sufficiently large $n$. Since
$T|\mathcal{H}_{n}$ has cyclic multiplicity equal to one when $\mathcal{H}_{n}\ne\{0\}$,
we can select for each $n$ a vector $x_{n}\in\mathcal{H}_{n}$ which
is cyclic for $T|\mathcal{H}_{n}$. A collection $\{x_{n}\}$ obtained
this way will be called a $C_{0}$-\emph{basis} for the operator $T$.
The vector $x_{1}$ is also known as a \emph{maximal vector }for $T$.
Just like a linearly independent set in a vector space can be completed
to a basis, a partial $C_{0}$-basis can be completed to a $C_{0}$-basis.
We formulate the result for the case of finite multiplicity, in an
essentially equivalent form. The proof is contained in \emph{\cite[Theorem III.6.10]{C0-book}.
}The last statement follows from the main result of\emph{ \cite{three-test-pro}.}
\begin{prop}
\label{prop:completion-of-basis}Let $T$ be an operator of class
$C_{0}$ on $\mathcal{H}$. Assume that $k$ is a positive integer,
and that $\mathcal{M\subset\mathcal{H}}$ is an invariant subspace
for $T$ such that $T|\mathcal{M}$ is quasisimilar to 
\[
\bigoplus_{n=1}^{k}S(\theta_{n}^{T}).
\]
Then $\mathcal{M}$ has a $T$-invariant quasidirect complement. If
$\mathcal{N}$ is such a complement of $\mathcal{M}$, the Jordan
model of $T|\mathcal{N}$ is
\[
\bigoplus_{n\ge1}S(\theta_{n+k}^{T}).
\]
If two $T$-invariant subspaces $\mathcal{M},\mathcal{N}$ are quasidirect
complements of each other and if $T|\mathcal{M}$ is quasisimilar
to 
\[
\bigoplus_{j=1}^{k}S(\theta_{n_{j}}^{T}),
\]
where $1\le n_{1}<n_{2}<\cdots<n_{k}<\aleph_{0}$, then $T|\mathcal{N}$
is quasisimilar to
\[
\bigoplus_{n\notin\{n_{1},n_{2},\dots,n_{k}\}}S(\theta_{n}^{T}).
\]

\end{prop}
In the case of operators of finite multiplicity it is somewhat easier
to verify that a system of vectors is a $C_{0}$-basis.
\begin{prop}
\label{prop:basis when multiplicity is finite}Let $T$ be an operator
of class $C_{0}$ on $\mathcal{H}$ and let $\bigoplus_{n=1}^{N}S(\theta_{n})$
be the Jordan model of $T$. Assume that the vectors $\{h_{n}\}_{n=1}^{N}\subset\mathcal{H}$
satisfy the following properties:
\begin{enumerate}
\item The smallest invariant subspace for $T$ containing $\{h_{n}\}_{n=1}^{N}$
is $\mathcal{H}$.
\item $\theta_{n}(T)h_{n}=0$. 
\end{enumerate}

Then $\{h_{n}\}_{n=1}^{N}$ is a $C_{0}$-basis for $T$.

\end{prop}
\begin{proof}
Denote by $\mathcal{H}_{n}$ the cyclic space for $T$ generated by
$h_{n}$. Condition (2) implies that $T|\mathcal{H}_{n}\sim S(\varphi_{n})$
for some inner divisor $\varphi_{n}$ of $\theta_{n}$. Choose injective
operators with dense ranges $X_{n}:\mathcal{H}(\varphi_{n})\to\mathcal{H}_{n}$
such that 
\[
(T|\mathcal{H}_{n})X_{n}=X_{n}S(\varphi_{n}),\quad n=1,2,\dots,N,
\]
and define $X:\bigoplus_{n=1}^{N}\mathcal{H}(\varphi_{n})\to\mathcal{H}$
by
\[
X(u_{1}\oplus u_{2}\oplus\cdots\oplus u_{N})=\sum_{n=1}^{N}X_{n}u_{n},\quad u_{1}\oplus u_{2}\oplus\cdots\oplus u_{N}\in\bigoplus_{n=1}^{N}\mathcal{H}(\varphi_{n}).
\]
Then $X$ has dense range, and \cite[Theorem VI.3.16]{C0-book} implies
that $\prod_{n=1}^{N}\theta_{n}$ divides $\prod_{n=1}^{N}\varphi_{n}$.
It follows that $\varphi_{n}\equiv\theta_{n}$, $n=1,2,\dots,N$,
and a second application of \cite[Theorem VI.3.16]{C0-book} implies
that $X$ is one-to-one as well. In fact, $X$ implements an isomorphism
between the lattices of invariant subspaces of $\bigoplus_{n=1}^{N}S(\varphi_{n})=\bigoplus_{n=1}^{N}S(\theta_{n})$
and $T$ via the map $\mathcal{M}\mapsto\overline{X\mathcal{M}}$,
and this implies immediately the conclusion of the proposition.
\end{proof}
There is a natural way to transform a $C_{0}$-basis into another.
We recall that the algebraic adjoint (or cofactor matrix) of an $N\times N$
matrix $A$ is a matrix $A^{{\rm Ad}}$ such that
\[
AA^{{\rm Ad}}=A^{{\rm Ad}}A=\det(A)I_{N},
\]
where $I_{N}$ denotes the identity matrix of size $N$. Given functions
$u,v\in H^{\infty}$, at least one of which is nonzero, we denote
by $u\wedge v$ their greatest common inner divisor.
\begin{cor}
\label{cor:base change}Let $T$ be an operator of class $C_{0}$
on $\mathcal{H}$ with Jordan model $\bigoplus_{n=1}^{N}S(\theta_{n})$,
and let $\{h_{n}\}_{n=1}^{N}\subset\mathcal{H}$ be a $C_{0}$-basis
for $T$. Consider functions $\{u_{nm}\}_{n,m=1}^{N}\subset H^{\infty}$
with the following properties:
\begin{enumerate}
\item $\theta_{1}\wedge\det[u_{nm}]_{n,m=1}^{N}\equiv1$.
\item $\theta_{m}/\theta_{n}$ divides $u_{nm}$ if $n>m$.
\end{enumerate}

Then the vectors 
\[
h'_{n}=\sum_{m=1}^{N}u_{nm}(T)h_{m},\quad n=1,2,\dots,N,
\]
also form a $C_{0}$-basis for $T$.

\end{cor}
\begin{proof}
Denote by $[v_{nm}]_{n,m=1}^{N}$ the algebraic adjoint of the matrix
$[u_{nm}]_{n,m=1}^{N}$, and set $g=\det[u_{nm}]_{n,m=1}^{N}$ so
that
\[
\sum_{j=1}^{N}v_{nm}u_{mk}=g\delta_{ik},\quad i,k=1,2,\dots,N.
\]
We have
\[
\sum_{i=1}^{N}v_{nm}(T)h'_{m}=g(T)h_{n},\quad n=1,2,\dots,N,
\]
so that the invariant subspace for $T$ generated by $\{h'_{n}\}_{n=1}^{N}$
contains $g(T)\mathcal{H}$. Condition (1) implies that $g(T)$ has
dense range, and therefore this invariant subspace is $\mathcal{H}$.
The corollary follows from Proposition \ref{prop:basis when multiplicity is finite}
once we show that $\theta_{m}(T)h_{m}'=0$. Indeed, since $\theta_{k}|\theta_{m}$
for $k>m$,
\[
\theta_{m}(T)h'_{m}=\sum_{n=1}^{m-1}u_{mn}(T)\theta_{m}(T)h_{n}.
\]
Condition (2) implies that $\theta_{n}|u_{mn}\theta_{m}$ and therefore
all the terms in the sum above vanish.
\end{proof}
There is yet another quasidirect decomposition which serves as a substitute
for the primary decomposition of torsion modules over a principal
ideal domain. The following result is easily obtained from \cite[Theorem II.4.6]{C0-book}.
\begin{prop}
\label{prop:primary-decomposition}Consider an operator $T$ of class
$C_{0}$ on a Hilbert space $\mathcal{H}$. Assume that the minimal
function $m_{T}$ is factored as a product
\[
m_{T}=\gamma_{1}\gamma_{2}\cdots\gamma_{n}
\]
of inner functions such that $\gamma_{i}\wedge\gamma_{j}\equiv1$
for $i\ne j$, and set $\Gamma_{j}=m_{T}/\gamma_{j}$ for $j=1,2,\dots,n$.
Then $\mathcal{H}$ is the quasidirect sum of the invariant subspaces
\[
\mathcal{H}_{j}=\overline{\Gamma_{j}(T)\mathcal{H}},\quad j=1,2,\dots,n.
\]
If $\mathcal{M}$ is an invariant subspace for $T$, we have
\[
\overline{\Gamma_{j}(T)\mathcal{M}}=\mathcal{M}\cap\overline{\Gamma_{j}(T)\mathcal{H}},\quad j=1,2,\dots,n,
\]
and $\mathcal{M}$ is the quasidirect sum of the spaces $\mathcal{M}_{j}=\overline{\Gamma_{j}(T)\mathcal{M}}$.
Moreover, $\mathcal{M}$ has an invariant quasidirect complement in
$\mathcal{H}$ if and only if each $\mathcal{M}_{j}$ has an invariant
quasidirect complement in $\mathcal{H}_{j}$.
\end{prop}
The decomposition of $m_{T}$ to which this proposition is applied
arise as follows.
\begin{lem}
\label{lem:achieving-total order}
\begin{enumerate}
\item Consider functions $m,f_{1},f_{2},\dots,f_{k}\in H^{\infty}$ such
that $m$ is inner. There exist pairwise relative prime inner functions
$\gamma_{1},\gamma_{2},\dots,\gamma_{n}$ in $H^{\infty}$ with the
property that $m=\gamma_{1}\gamma_{2}\cdots\gamma_{n}$ and the set
of inner functions 
\[
\{f_{1}\wedge\gamma_{i},f_{2}\wedge\gamma_{i},\dots,f_{k}\wedge\gamma_{i}\}
\]
 is totally ordered by divisibility for $i=1,2,\dots,n$.
\item Consider inner functions $\theta_{1},\theta_{2},\dots,\theta_{N}\in H^{\infty}$
such that $\theta_{j+1}|\theta_{j}$ for $j=1,2,\dots,N-1$. There
exists a factorization $\theta_{1}=\gamma_{1}\gamma_{2}\cdots\gamma_{n}$
into pairwise relatively prime inner factors with the following property:
if $\omega$ is a nonconstant inner factor of $\gamma_{k}$ for some
$k=1,2,\dots,n$ and $\theta_{j}\wedge\gamma_{k}$ is not constant
for some $j=2,\dots,N$, then
\[
\omega\wedge\theta_{j}\wedge\gamma_{k}\not\equiv1.
\]

\end{enumerate}
\end{lem}
\begin{proof}
For the proof, it suffices to consider the case when $m$ is either
a Blaschke product or a singular inner function. We only treat the
case when 
\[
m(\lambda)=\exp\left[\int_{0}^{2\pi}\frac{\lambda+e^{it}}{\lambda-e^{it}}\, d\mu(t)\right],\quad\lambda\in\mathbb{D},
\]
where $\mu$ is a singular measure on the interval $[0,2\pi)$. We
have then
\[
(f_{j}\wedge m)(\lambda)=\exp\left[\int_{0}^{2\pi}\frac{\lambda+e^{it}}{\lambda-e^{it}}h_{j}(t)\, d\mu(t)\right],\quad\lambda\in\mathbb{D},j=1,2\dots,k,
\]
where the functions $h_{j}:[0,2\pi)\to[0,1]$ are Borel measurable.
There is a Borel partition $[0,2\pi)=\bigcup_{\sigma}A_{\sigma}$
indexed by the permutations $\sigma$ of $\{1,2,\dots,k\}$ such that
\[
h_{\sigma(1)}(t)\le h_{\sigma(2)}(t)\le\cdots\le h_{\sigma(k)}(t),\quad t\in A_{\sigma},
\]
for each $\sigma$. We write the function $m$ as the product of the
inner functions
\[
\gamma_{\sigma}(\lambda)=\exp\left[\int_{A_{\sigma}}\frac{\lambda+e^{it}}{\lambda-e^{it}}\, d\mu(t)\right],\quad\lambda\in\mathbb{D}.
\]
This decomposition satisfies the conclusion of (1). For (2), we assume
again that the functions $\theta_{j}$ are singular, so that
\[
\theta_{j}(\lambda)=\exp\left[\int_{0}^{2\pi}\frac{\lambda+e^{it}}{\lambda-e^{it}}h_{j}(t)\, d\mu(t)\right],\quad\lambda\in\mathbb{D},j=1,2\dots,N,
\]
where the Borel functions $h_{j}$ are such that 
\[
1=h_{1}\ge h_{2}\ge\cdots\ge h_{N}\ge0.
\]
Define now Borel sets $B_{k}=\{t\in[0,2\pi):h_{k}(t)>h_{k+1}(t)=0\}$,
$k=1,2,\dots,N$, using the convention $h_{N+1}=0$ in the definition
of $B_{N}$. The functions
\[
\gamma_{k}(\lambda)=\exp\left[\int_{B_{k}}\frac{\lambda+e^{it}}{\lambda-e^{it}}\, d\mu(t)\right],\quad\lambda\in\mathbb{D},k=1,2\dots,N,
\]
satisfy the requirements of (2) with $n=N$. The case of Blaschke
products is treated similarly, with the functions $h_{j}$ being replaced
by the functions $\nu_{j}(\lambda)$ representing the order of $\lambda\in\mathbb{D}$
as a zero of $\theta_{j}$.
\end{proof}
We need more precise information about the construction of the Jordan
model for operators $T$ of class $C_{0}$ for which $I-T^{*}T$ has
finite rank $N$. Such operators are said to be of class $C_{0}(N)$;
they are constructed as follows. Consider an inner function $\Theta$
on the unit disk whose values are complex $N\times N$ matrices. Thus,
$\|\Theta(\lambda)\|\le1$ for every $\lambda\in\mathbb{D}$, and
the the boundary values $\Theta(\zeta)$ are unitary matrices for
almost every $\zeta\in\mathbb{T}=\partial\mathbb{D}$. Such a function
determines a multiplication operator $M_{\Theta}$ on the space $H^{2}\otimes\mathbb{C}^{N}$,
and the space $M_{\Theta}[H^{2}\otimes\mathbb{C}^{N}]$ is invariant
for the shift $S\otimes I_{\mathbb{C}^{N}}$ of multiplicity $N$.
As in the case $N=1$, we write
\[
\mathcal{H}(\Theta)=[H^{2}\otimes\mathbb{C}^{N}]\ominus M_{\Theta}[H^{2}\otimes\mathbb{C}^{N}],
\]
and denote by $S(\Theta)$ the compression of $S\otimes I_{\mathbb{C}^{N}}$
to $\mathcal{H}(\Theta)$. The operator $S(\Theta)$ constructed this
way is of class $C_{0}(N')$ for some $N'\le N$, and every operator
of class $C_{0}(N)$ is unitarily equivalent to $S(\Theta)$ for some
function $\Theta$ with the above properties. The Jordan model of
$S(\Theta)$ can be obtained directly by finding an analogue of the
Smith normal form for the matrix $\Theta$ \cite{nordg,moore-nor}.
We recall the basic definitions. Assume that $A$ and $B$ are two
$p\times q$ matrices with elements from $H^{\infty}$. We say that
$A$ is \emph{quasiequivalent} to $B$ if, for any inner function
$\omega$, there exist matrices $X,Y$ over $H^{\infty}$ of sizes
$p\times p,q\times q$, respectively, such that
\begin{equation}
XA=BY\label{eq:quaquasi}
\end{equation}
and the functions $\det(X),\det(Y)$ are relatively prime to $\omega$,
that is, neither of them has any nonconstant common inner factors
with $\omega$. Despite the asymmetry in the definition, quasiequivalence
is an equivalence relation. Indeed, equation (\ref{eq:quaquasi})
implies the relation
\[
X_{1}B=AY_{1},
\]
where $X_{1}=\det(Y)X^{{\rm Ad}}$ and $Y_{1}=\det(X)Y^{{\rm Ad}}$.
According to \cite{nordg}, every $p\times q$ matrix $A$ over $H^{\infty}$
is quasiequivalent to a matrix of the form
\[
\left[\begin{array}{ccccc}
\theta_{1} & 0 & 0 & \cdots & 0\\
0 & \theta_{2} & 0 & \cdots & 0\\
0 & 0 & \theta_{3} & \cdots & 0\\
\vdots & \vdots & \vdots & \ddots & \vdots\\
0 & 0 & 0 & \cdots
\end{array}\right],
\]
where the functions $\theta_{1},\theta_{2},\dots,\theta_{\min\{p,q\}}$
are inner or zero and satisfy $\theta_{n+1}|\theta_{n}$. These functions
are uniquely determined except for scalar factors of absolute value
$1$. None of the functions $\theta_{n}$ is zero if $A$ has a nonzero
minor of order $\min\{p,q\}$. This result can be applied to an inner
$N\times N$ matrix $\Theta$ to yield inner functions $\theta_{1},\theta_{2},\dots,\theta_{N}$
such that $\theta_{n+1}|\theta_{n}$ for $n=1,2,\dots,N-1$, and $N\times N$
matrices $X,Y$ over $H^{\infty}$ such that 
\[
\Theta X=Y\Theta'
\]
and $\det(X),\det(Y)$ are relatively prime to $\theta_{1}$, where
$\Theta'$ is the diagonal matrix with diagonal entries $\theta_{1},\theta_{2},\dots,\theta_{N}$.
The conditions on the determinants above can be written as
\[
\det(X)\wedge\theta_{1}\equiv\det(Y)\wedge\theta_{1}\equiv1.
\]

Denote by $y_{1},y_{2},\dots,y_{N}$ the columns of the matrix $Y$,
which can be viewed as vectors in $H^{2}\otimes\mathbb{C}^{N}$. In
other words, $y_{n}=Y(1\otimes e_{n})$, where $\{e_{n}\}_{n=1}^{N}$
is the standard basis in $\mathbb{C}^{N}$. Then the results of \cite{moore-nor}
say that $\theta_{n}^{S(\Theta)}\equiv\theta_{n}$ for $n=1,2,\dots,N,$
and the vectors $\{P_{\mathcal{H}(\Theta)}y_{n}\}_{n=1}^{N}$ form
a $C_{0}$-basis for $S(\Theta)$. Note incidentally that $S(\Theta)$
need not have multiplicity equal to $N$. Indeed, the last few of
the functions $\theta_{n}$ could be constant, and the corresponding
vectors in the $C_{0}$-basis would be $0$. The following lemma provides
a formulation of the $C_{0}$-basis property in terms of the vectors
$y_{n}$.
\begin{lem}
\label{lem:troubling-lemma-1}With the above notation, assume that
$\{u_{n}\}_{n=1}^{N}\subset H^{\infty}$ and 
\begin{equation}
\sum_{n=1}^{N}u_{n}y_{n}\in M_{\Theta}[H^{2}\otimes\mathbb{C}^{N}].\label{eq:trouble1-1}
\end{equation}
Then $\theta_{n}|u_{n}$ for $n=1,2,\dots,N$. \end{lem}
\begin{proof}
Observe that
\[
\sum_{n=1}^{N}u_{n}y_{n}=\sum_{n=1}^{N}Y(u_{n}\otimes e_{n})=Y\Theta'\sum_{n=1}^{N}\frac{u_{n}}{\theta_{n}}\otimes e_{j}=\Theta X\sum_{n=1}^{N}\frac{u_{n}}{\theta_{n}}\otimes e_{n},
\]
and (\ref{eq:trouble1-1}) implies that
\[
X\sum_{n=1}^{N}\frac{u_{n}}{\theta_{n}}\otimes e_{n}\in H^{2}\otimes\mathbb{C}^{N}
\]
 because $M_{\Theta}$ is one-to-one. Therefore
\[
\det(X)\sum_{n=1}^{N}\frac{u_{n}}{\theta_{n}}\otimes e_{n}=X^{{\rm Ad}}X\sum_{n=1}^{N}\frac{u_{n}}{\theta_{n}}\otimes e_{n}\in H^{2}\otimes\mathbb{C}^{N}
\]
 as well, so that $\theta_{n}|u_{n}\det(X)$ for all $n$. Since $\theta_{n}\wedge\det(X)|\theta_{1}\wedge\det(X)\equiv1$,
we conclude that $\theta_{n}|u_{n}$, as claimed.
\end{proof}
We use repeatedly the following result about operators of class $C_{0}$
with finite multiplicity. The proof follows from \cite[Proposition VII.6.9]{C0-book}.

\begin{prop}
\label{prop:dense-vectors}Let $T$ be an operator of class $C_{0}$
on $\mathcal{H}$ with $\mu_{T}<+\infty$, let $T'$ be a completely
nonunitary contraction on $\mathcal{H}'$, and let $X:\mathcal{H}'\to\mathcal{H}$
be a linear operator with dense range such that $XT'=TX$. Then every
invariant subspace $\mathcal{M}$ for $T$ is of the form $\overline{X\mathcal{M}'}$,
where $\mathcal{M}'$ is an invariant subspace for $T'$. If $T|\mathcal{M}$
has a cyclic vector, then it also has a cyclic vector of the form
$Xh'$ with $h'\in\mathcal{H}'.$
\end{prop}
It is useful to consider more general invariant subspaces of $S\otimes I_{\mathbb{C}^{N}}$.
These are characterized by the Beurling-Lax-Halmos theorem \cite[Theorem V.3.3]{SN}.
\begin{thm}
\label{thm:Beurling-L-H}Consider an invariant subspace $\mathcal{K}\subset H^{2}\otimes\mathbb{C}^{N}$
for $S\otimes I_{\mathbb{C}^{N}}.$ There exist an integer $r\le N$
and an inner function $\Psi$ with values $N\times r$ complex matrices
so that $\mathcal{K}=M_{\Psi}[H^{2}\otimes\mathbb{C}^{r}]$.
\end{thm}
The fact that $\Psi$ is inner implies that it has nonzero minors
of order $r$, and therefore quasidiagonalization produces a matrix
with $r$ inner functions $\psi_{1},\psi_{2},\dots,\psi_{r}$ on the
main diagonal. We call the number $r$ the rank of the invariant subspace
$\mathcal{K}$, and observe that $r$ is simply the multiplicity of
the unilateral shift $S\otimes I_{\mathbb{C}^{N}}|\mathcal{K}$. We
also use the notation 
\[
d(\mathcal{K})=\psi_{1}\psi_{2}\cdots\psi_{r}
\]
 for the product of these functions. The function $d(\mathcal{K})$
is inner, and it is uniquely determined up to a scalar factor. In
the special case when $\mathcal{K}$ is of maximum rank $r=N$, we
have 
\[
d(\mathcal{K})\equiv\det(\Psi).
\]

More generally, if $V$ is a unilateral shift of finite multiplicity
on a space $\mathcal{M}$ and $\mathcal{K}\subset\mathcal{M}$ is
an invariant subspace for $V$, we can define an inner function
\[
d_{\mathcal{M}}(\mathcal{K})
\]
by noting that $V$ is unitarily equivalent to $S\otimes I_{\mathbb{C}^{N}}$
for some $N$, and identifying $\mathcal{K}$ with the range of an
inner function as above. The multiplicative property of determinants
implies that
\[
d_{\mathcal{M}}(\mathcal{L})\equiv d_{\mathcal{M}}(\mathcal{K})d_{\mathcal{K}}(\mathcal{L})
\]
if $\mathcal{L}\subset\mathcal{K}$ are invariant subspaces of rank
$N$ of $V$.
\begin{lem}
\label{lem:when-d_M(K)=00003D1}Let $V$ be a unilateral shift of
finite multiplicity on a space $\mathcal{M}$, and let $\mathcal{K}\subset\mathcal{M}$
be an invariant subspace. The following conditions are equivalent:
\begin{enumerate}
\item $d_{\mathcal{M}}(\mathcal{K})\not\equiv1$.
\item There exist an inner function $\varphi\in H^{\infty}$ and a vector
$f\in\mathcal{M}\setminus\mathcal{K}$ such that $\varphi(V)f\in\mathcal{K}$.
\end{enumerate}
\end{lem}
\begin{proof}
We assume without loss of generality that $V=S\otimes I_{\mathbb{C}^{N}}$
for some $N\in\mathbb{N}$, and that $\mathcal{K}=M_{\Psi}[H^{2}\otimes\mathbb{C}^{r}]$
for some $r\in\{1,2,\dots,N\}$ and some inner $N\times r$ matrix
$\Psi$. Choose square matrices $X,Y$ over $H^{\infty}$ such that
$\det(X)\wedge d_{\mathcal{M}}(\mathcal{K})\equiv\det(Y)\wedge d_{\mathcal{M}}(\mathcal{K})\equiv1$
and $\Psi X=Y\Psi'$, where $\Psi'$ has inner entries $\psi_{1},\psi_{2},\dots,\psi_{r}$
on the main diagonal, zero entries elsewhere, and $\psi_{j+1}|\psi_{j}$
for $j=1,2,\dots,r-1$. Denote by $y_{1},y_{2},\dots,y_{r}$ the columns
of the matrix $Y$, and denote by $e_{1},e_{2},\dots,e_{r}$ the standard
basis in $\mathbb{C}^{r}$, which we also view as a subspace of $\mathbb{C}^{N}$.
Condition (1) is equivalent to $\psi_{1}\not\equiv1$. Assume first
that $\psi_{1}\not\equiv1$, and observe that
\[
\psi_{1}(V)y_{1}=\psi_{1}y_{1}=Y(\psi_{1}\otimes e_{1})=Y\Psi'(1\otimes e_{1})=\Psi X(1\otimes e_{1})
\]
belongs to the space $\mathcal{K}$, but $y_{1}\notin\mathcal{K}$,
as can be seen by repeating the proof of Lemma \ref{lem:troubling-lemma-1}.
Thus (1) implies (2). Conversely, assume that $\psi_{1}\equiv1$ and
a vector $f\in H^{2}\otimes\mathbb{C}^{N}$ satisfies $\varphi f\in\mathcal{K}$
for some inner function $\varphi$, say $\varphi f=\Psi h$ for some
$h\in H^{2}\otimes\mathbb{C}^{r}$. Note that the matrices $X,Y$
above can now be chosen so that their determinants are relatively
prime to $\varphi$, and $\Psi'$ can be chosen to simply be the matrix
representing the inclusion $\mathbb{C}^{r}\subset\mathbb{C}^{N}$.
We have 
\[
\varphi\det(X)f=\Psi\det(X)h=\Psi XX^{{\rm Ad}}h=Y\Psi'X^{{\rm Ad}}h,
\]
so that multiplying by $Y^{{\rm Ad}}$ yields 
\[
\varphi\det(X)Y^{{\rm Ad}}f=\det(Y)\Psi'X^{{\rm Ad}}h.
\]
 Apply now the matrix $X\oplus I_{\mathbb{C}^{N-r}}$ to both sides
to obtain
\[
(X\oplus I_{\mathbb{C}^{N-r}})\det(X)Y^{{\rm Ad}}\varphi f=\det(X)\det(Y)\Psi'h.
\]
In other words, since $\varphi\wedge(\det(X)\det(Y))\equiv1$, $\varphi$
divides all the components of the vector $\Psi'h$, and this simply
means that $h/\varphi\in H^{2}\otimes\mathbb{C}^{r}$. We conclude
that $f=\Psi(h/\varphi)$ does belong to $\mathcal{K}$, thus showing
that property (2) does not hold. Thus (2) implies (1).\end{proof}
\begin{cor}
Let $V$ be a shift of finite multiplicity on a Hilbert space $\mathcal{M}$,
and let $\mathcal{K}$ be an invariant subspace for $V$ such that
$d_{\mathcal{M}}(\mathcal{K})\equiv1$ and $V|\mathcal{K}$ has multiplicity
$r$. Fix an inner function $\omega\in H^{\infty}$. There exist vectors
$h_{1},h_{2},\dots,h_{r}\in\mathcal{K}$ with the following property:
if the functions $u_{1},u_{2},\dots,u_{r}\in H^{\infty}$ are such
that $\sum_{j=1}^{r}u_{j}(V)h_{j}\in\omega(V)\mathcal{M}$, then $\omega|u_{j}$
for all $j=1,2,\dots,r$.\end{cor}
\begin{proof}
We assume with no loss of generality that $V=S\otimes I_{\mathbb{C}^{N}}$
on $\mathcal{M}=H^{2}\otimes\mathbb{C}^{N}$, and $\mathcal{Q}=M_{\Psi}[H^{2}\otimes\mathbb{C}^{r}]$
for some inner function $\Psi$. We apply quasiequivalence to obtain
square matrices $X,Y$ over $H^{\infty}$ such that $\Psi X=Y\Psi'$
and $\det(X)\wedge\omega\equiv\det(Y)\wedge\omega\equiv1$, where
$\Psi'$ is the matrix representing the inclusion $\mathbb{C}^{r}\subset\mathbb{C}^{N}$.
We claim that the columns $h_{1},h_{2},\dots,h_{r}$ of the matrix
$Y\Psi'$ satisfy the required property. Indeed, the equation $h_{j}=Y\Psi'(1\otimes e_{j})=\Psi X(1\otimes e_{j})$
shows that $h_{j}\in\mathcal{K}$. Assume now that $\sum_{j=1}^{r}u_{j}h_{j}=\omega h$
for some $h\in H^{2}\otimes\mathbb{C}^{N}$, where the coefficients
$u_{j}$ belong to $H^{\infty}$. We have
\[
\det(Y)\sum_{j=1}^{r}u_{j}\otimes e_{j}=Y^{{\rm Ad}}Y\sum_{j=1}^{r}u_{j}\otimes e_{j}=Y^{{\rm Ad}}\sum_{j=1}^{r}u_{j}h_{j}=\omega Y^{{\rm Ad}}h,
\]
and the divisibility $\omega|u_{j}$ follows because $\det(Y)\wedge\omega\equiv1$.
\end{proof}
Subspaces $\mathcal{K}$ with $d_{\mathcal{M}}(\mathcal{K})\equiv1$
are obtained as follows. Consider the field of fractions $\mathfrak{D}$
of $H^{\infty}$. That is $\mathfrak{D}=\{\varphi/\psi:\varphi\in H^{\infty},\psi\in H^{\infty}\setminus\{0\}\}$;
recall that $H^{2}\subset\mathfrak{D}$. A unilateral shift $V$ of
finite multiplicity on $\mathcal{M}$ turns $\mathcal{M}$ into a
module over $H^{\infty}$, and this module is contained in the finite
dimensional vector space $\mathfrak{D\mathcal{M}}$ over $\mathfrak{D}$.
Indeed, $\mathfrak{D}(H^{2}\otimes\mathbb{C}^{N})=\mathfrak{D}\otimes\mathbb{C}^{N}=\mathfrak{D}^{N}$,
thus showing that $\dim_{\mathfrak{D}}(\mathfrak{D}\mathcal{M})$
equals the multiplicity of $V$.
\begin{lem}
\label{lem:vector-space-intersect-H2}Let $V$ be a shift of finite
multiplicity on a Hilbert space $\mathcal{M}$, and let $\mathcal{Q}\subset\mathfrak{D}\mathcal{M}$
be a $\mathfrak{D}$-vector subspace. Then $\mathcal{K}=\mathcal{Q}\cap\mathcal{M}$
is closed in $\mathcal{M}$, it is invariant for $V$, and $d_{\mathcal{M}}(\mathcal{K})\equiv1$.
Conversely, every invariant subspace $\mathcal{K\subset\mathcal{M}}$
such that $d_{\mathcal{M}}(\mathcal{K})\equiv1$ satisfies the relation
$\mathcal{K}=(\mathfrak{D}\mathcal{K})\cap\mathcal{M}$.\end{lem}
\begin{proof}
We assume without loss of generality that $V=S\otimes I_{\mathbb{C}^{N}}$
so that 
\[
\mathfrak{D}\mathcal{M}=\mathfrak{D}\otimes\mathbb{C}^{N}=\left\{ \sum_{j=1}^{N}u_{j}\otimes e_{j}:u_{1},u_{2},\dots,u_{N}\in\mathfrak{D}\right\} .
\]
The vector space $\mathcal{Q}$ is defined by a finite number of linear
equations of the form
\begin{equation}
\sum_{j=1}^{N}\alpha_{j}u_{j}=0,\label{eq:linear-combo}
\end{equation}
with coefficients $\alpha_{j}\in\mathfrak{D}$. The solution set of
such an equation is not modified if we multiply all the coefficients
by the same function in $H^{\infty}\setminus\{0\}$. We can thus assume
that $\alpha_{j}\in H^{\infty}$ for all $j$. It follows that $\mathcal{\mathcal{K}=Q}\cap(H^{2}\otimes\mathbb{C}^{N})$
consists of those vectors $\sum_{j=1}^{N}u_{j}\otimes e_{j}$ for
which the functions $u_{j}\in H^{2}$ satisfy all the equations (\ref{eq:linear-combo})
defining $\mathcal{Q}$, an therefore $\mathcal{K}$ is indeed closed.
Finally, $\mathcal{K}$ is easily seen not to satisfy property (2)
of Lemma \ref{lem:when-d_M(K)=00003D1}.

Assume now that $d_{\mathcal{M}}(\mathcal{K})\equiv1$ for some invariant
space $\mathcal{K}\subset H^{2}\otimes\mathbb{C}^{N}$. Clearly, $\mathcal{K}\subset(\mathfrak{D}\mathcal{K})\cap\mathcal{M}$,
so we it suffices prove the opposite inclusion. Consider a vector
$h\in(\mathfrak{D}\mathcal{K})\cap\mathcal{M}$, so there exist $\varphi\in H^{\infty}$,
$\psi\in H^{\infty}\setminus\{0\}$, and $k\in\mathcal{K}$ such that
$h=(\varphi/\psi)k$. Factor $\psi=\psi_{1}\psi_{2}$, where $\psi_{1}$
is inner and $\psi_{2}$ is outer. We have
\[
\psi_{1}\psi_{2}h=\psi h=\varphi k\in\mathcal{K},
\]
and therefore $\psi_{2}h\in\mathcal{K}$ by Lemma \ref{lem:when-d_M(K)=00003D1}.
Since $\psi_{2}$ is outer, there exists a sequence $\{u_{n}\}_{n=1}^{\infty}\subset H^{\infty}$
such that $\lim_{n\to\infty}u_{n}\psi_{2}=1$ in the weak{*} topology
of $H^{\infty}$. This implies that $\lim_{n\to\infty}u_{n}\psi_{2}h=h$
in $H^{2}\otimes\mathbb{C}^{N}$, and therefore $h\in\mathcal{K}$. 
\end{proof}
The following result is useful when we want to replace a linear combination
with coefficients in $\mathfrak{D}$ into another one with coefficients
in $H^{\infty}$. It is a stronger property which cyclic vectors for
an operator of class $C_{0}$ have.
\begin{lem}
\label{lem:vectors in the space of a cyclic operator}Consider an
operator $T$ of class $C_{0}$ on a Hilbert space $\mathcal{H}$,
a cyclic vector $h\in\mathcal{H}$ for $T$, and an inner function
$\omega$. For every vector $k\in\mathcal{H}$ there exist functions
$\alpha,\beta\in H^{\infty}$ such that $\beta\wedge\omega\equiv1$
and $\alpha(T)h=\beta(T)k$.\end{lem}
\begin{proof}
Let $S(\theta)$ be the Jordan model of $T$, and let $X:\mathcal{H}\to\mathcal{H}(\theta)$
be an injective operator with dense range such that $XT=S(\theta)X$.
If $\alpha,\beta\in H^{\infty}$, we have
\[
\alpha(S(\theta))Xh=X\alpha(T)h,\quad\beta(S(\theta))Xk=X\beta(T)k.
\]
Th equality $\alpha(T)h=\beta(T)k$ is equivalent to 
\begin{equation}
\alpha(S(\theta))Xh=\beta(S(\theta))Xk\label{eq:an equation}
\end{equation}
because $X$ is one-to-one. Since $Xh$ is a cyclic vector for $S(\theta)$,
this observation shows that it suffices to prove the lemma when $T=S(\theta)$.
In this case the functions $h,k\in H^{2}$ can be written as $h=\beta/\gamma$
and $k=\alpha/\gamma$ for some functions $\alpha,\beta,\gamma\in H^{\infty}$
such that $\gamma$ is outer. The equation (\ref{eq:an equation})
is satisfied with this choice of $\alpha$ and $\beta$. Moreover,
the fact that $h$ is a cyclic vector for $S(\theta)$ amounts to
the equality $\beta\wedge\theta\equiv1$. Equation (\ref{eq:an equation})
remains valid if $\beta$ is replaced by $\beta+t\theta$ for some
scalar $t\in\mathbb{C}$. It is known \cite[Theorem III.1.14]{C0-book}
that for $t$ in a dense $G_{\delta}$ set, we have $(\beta+t\theta)\wedge\omega\equiv1$.
The lemma follows.
\end{proof}
We continue with one useful result from intersection theory. Consider
an arbitrary field $\mathfrak{D}$ and a vector space $L$ over $\mathfrak{D}$
of dimension $N<+\infty$. Given an integer $r\in\{1,2,\dots,N\}$,
the \emph{Grassmann variety} $G(L,r)$ consists of all vector subspaces
$M\subset L$ of dimension $r$. A \emph{complete flag} in $L$ is
a collection
\[
\mathcal{E}=\{\mathcal{E}_{0}\subset\mathcal{E}_{1}\subset\cdots\subset\mathcal{E}_{N}\}
\]
of subspaces of $L$ such that $\dim_{\mathfrak{D}}(\mathcal{E}_{i})=i$
for $i=0,1,\dots,N$. Assume that $\mathcal{E}$ is a complete flag
and 
\[
I=\{i_{1}<i_{2}<\cdots<i_{r}\}
\]
 is a subset of $\{1,2,\dots,N\}$. The \emph{Schubert variety} $\mathfrak{S}(\mathcal{E},I)$
is the subset of $G(L,r)$ consisting of those subspaces $M$ for
which
\[
\dim_{\mathfrak{D}}(M\cap\mathcal{E}_{i_{x}})\ge x,\quad x=1,2,\dots,r.
\]
The Littlewood-Richardson rule provides a test to determine whether
the intersection of three Schubert varieties $\mathfrak{S}(\mathcal{E},I)$,
$\mathfrak{S}(\mathcal{F},J)$, and $\mathfrak{S}(\mathcal{G},K)$
is nonempty. More precisely, assume that the three sets $I,J,K\subset\{1,2,\dots,N\}$
satisfy the equation
\[
\sum_{x=1}^{r}(i_{x}+j_{x}+k_{x}-3x)=2r(N-r).
\]
Then one associates to the triple $(I,J,K)$ a nonnegative integer
$c_{IJK}$ with the property that
\begin{equation}
\mathfrak{S}(\mathcal{E},I)\cap\mathfrak{S}(\mathcal{F},J)\cap\mathfrak{S}(\mathcal{G},K)\label{eq:intersection-of-3-Schubert-cells}
\end{equation}
 contains generically $c_{IJK}$ elements in case $\mathfrak{D}$
is algebraically closed. For general $\mathfrak{D}$, one can still
state that the intersection (\ref{eq:intersection-of-3-Schubert-cells})
contains at least one element when $c_{IJK}=1$; see \cite{bcdlt}.

\section{Invariant subspaces related to Schubert varieties\label{sec:Invariant-subspaces-vs-Schubert}}

In this section we fix an operator $T$ of class $C_{0}$ with finite
defect numbers. As seen above, we can assume that $T=S(\Theta)$,
where $\Theta$ is an inner function with values $N\times N$ matrices.
The Jordan model of $T$ has the form $\bigoplus_{n=1}^{N}S(\theta_{n})$,
where it may happen that the last few functions $\theta_{n}$ are
constant. Denote by $h_{1},h_{2},\dots,h_{N}$ a $C_{0}$-basis for
$T$, where $h_{n}=0$ if $\theta_{n}$ is constant. We study invariant
subspaces for $T$ of the form $[P_{\mathcal{H}(\Theta)}\mathcal{R}]^{-}$,
where $\mathcal{R}\subset H^{2}\otimes\mathbb{C}^{N}$ is an invariant
subspace for $S\otimes I_{\mathbb{C}^{N}}$ such that $d(\mathcal{R})=1$.
For the remainder of the paper, $\mathfrak{D}$ denotes the field
of fractions of $H^{\infty}$.
\begin{lem}
\label{lem:First preliminary for large space}Fix $m,r\in\{1,2,\dots,N\}$
such that $r\le m$. Assume that $\mathcal{Q}\subset\mathfrak{D}\otimes\mathbb{C}^{N}$
is a $\mathfrak{D}$-vector space of dimension $r$ contained in the
$\mathfrak{D}$-linear span of $\{h_{1},h_{2},\dots,h_{m}\}$, set
$\mathcal{R}=\mathcal{Q}\cap(H^{2}\otimes\mathbb{C}^{N})$, and $\mathcal{K}=[P_{\mathcal{H}(\Theta)}\mathcal{R}]^{-}$.
If the Jordan model of $T|\mathcal{K}$ is $\bigoplus_{n=1}^{r}S(\alpha_{n})$
then $\theta_{m}|\alpha_{r}$.\end{lem}
\begin{proof}
Observe that $(S\otimes I_{\mathbb{C}^{N}})|\mathcal{R}$ has cyclic
multiplicity $r$. Projecting onto $\mathcal{K}$ a cyclic set for
$(S\otimes I_{\mathbb{C}^{N}})|\mathcal{R}$ yields a cyclic set for
$T|\mathcal{K}$, and therefore $\mu_{T|\mathcal{R}}\le r$. Thus,
indeed, the Jordan model of $T|\mathcal{R}$ consists of at most $r$
summands. The lemma is vacuously satisfied if $\theta_{m}\equiv1$,
so we assume that is not the case. We show next that it suffices to
prove the lemma in the special case $m=N$. Denote by $\mathcal{E}$
the $\mathfrak{D}$-linear span of $\{h_{1},h_{2},\dots,h_{m}\}$,
and set 
\[
\mathcal{E}_{+}=\mathcal{E}\cap(H^{2}\otimes\mathbb{C}^{N}),\quad V=(S\otimes I_{\mathbb{C}^{N}})|\mathcal{E}_{+}.
\]
The operator $A=P_{\mathcal{H}(\Theta)}|\mathcal{E}_{+}$ satisfies
the intertwining relation
\[
AV=TA,
\]
and therefore the restriction $B=A|\mathcal{H}'$, where $\mathcal{H}'=\mathcal{E}_{+}\ominus\ker(A)$,
intertwines the compression $T'$ of $V$ to $\mathcal{H}'$ and $T$,
that is
\[
BT'=TB.
\]
Moreover, $B$ is one-to-one and its range contains $h_{1},h_{2},\dots,h_{m}$.
It is easy to see now that $T'\sim\bigoplus_{n=1}^{m}S(\theta_{n})$.
Indeed, $T'$ is an operator of class $C_{0}(m)$ so its Jordan model
contains at most $m$ summands $S(\varphi_{n})$, and $\varphi_{n}|\theta_{n}$
because $T'$ is quasisimilar to a restriction of $T$. On the other
hand, $\theta_{n}|\varphi_{n}$, $n=1,2,\dots,m$, because the restriction
of $T$ to its invariant subspace generated by $\{h_{1},h_{2},\dots,h_{m}\}$
is quasisimilar to a restriction of $T'$. Set now $\mathcal{R}'=\mathcal{Q}\cap\mathcal{E}_{+}$
and $\mathcal{K}'=[P_{\mathcal{H}'}\mathcal{R}']^{-}$. We have $B\mathcal{K}'\subset\mathcal{K}$
and therefore the Jordan model $\bigoplus_{n=1}^{r}S(\alpha'_{n})$
of $T'|\mathcal{K}'$ satisfies $\alpha'_{n}|\alpha_{n}$ for $n=1,2,\dots,r$.
It suffices therefore to show that $\theta_{m}|\alpha_{r}'$. This
concludes our reduction to the case $m=N$.

Represent the space $\mathcal{R}$ as
\[
\mathcal{R}=M_{\Psi}[H^{2}\otimes\mathbb{C}^{r}]
\]
for some inner $N\times r$ matrix $\Psi$. As noted in Lemma \ref{lem:vector-space-intersect-H2},
$\Psi$ is quasiequivalent to the constant matrix $J$ representing
the inclusion $\mathbb{C}^{r}\subset\mathbb{C}^{N}$. Fix square matrices
$X,Y$ over $H^{\infty}$ such that
\[
\Psi X=YJ
\]
and $\det(X)\wedge\theta_{N}\equiv\det(Y)\wedge\theta_{N}\equiv1$,
and denote by $y_{n}=Y(1\otimes e_{n})\in\mathcal{R}$, $n=1,2,\dots,r$,
the columns of $Y$. We claim that, given functions $u_{1},u_{2},\dots,u_{r}\in H^{\infty}$
such that
\[
\sum_{n=1}^{r}u_{n}(T)P_{\mathcal{H}(\Theta)}y_{n}=0,
\]
it follows that $\theta_{N}|u_{n}$ for $n=1,2,\dots,r$. Indeed,
the relation above is equivalent to $\sum_{n=1}^{r}u_{n}y_{n}\in M_{\Theta}[H^{2}\otimes\mathbb{C}^{N}]$.
We use now the fact that $\theta_{N}$ divides all the entries of
$\Theta$ to conclude that
\[
\frac{1}{\theta_{N}}YJ\left(\sum_{n=1}^{r}u_{n}\otimes e_{n}\right)=\frac{1}{\theta_{N}}\sum_{n=1}^{r}u_{n}y_{n}\in H^{2}\otimes\mathbb{C}^{N}.
\]
 Multiplying by $Y^{{\rm Ad}}$, we see that
\[
\frac{\det(Y)}{\theta_{N}}\sum_{n=1}^{r}u_{n}\otimes e_{n}\in H^{2}\otimes\mathbb{C}^{r}.
\]
The relation $\theta_{N}|u_{n}$ follows because $\det(Y)\wedge\theta_{N}\equiv1$.
The conclusion of the lemma follows by showing that $T|\mathcal{K}$
has a restriction quasisimilar to 
\[
\underbrace{S(\theta_{N})\oplus S(\theta_{N})\oplus\cdots\oplus S(\theta_{N})}_{r\text{ times}}.
\]
Indeed, denote by $\psi_{n}$ an inner function satisfying
\[
\psi_{n}H^{\infty}=\{\psi\in H^{\infty}:\psi(T)P_{\mathcal{H}(\Theta)}y_{n}=0\},\quad n=1,2,\dots,r.
\]
 The invariant subspace $\mathcal{K}'\subset\mathcal{K}$ for $T$
generated by the vectors 
\[
k_{n}=(\psi_{n}/\theta_{N})(T)P_{\mathcal{H}(\Theta)}y_{n},\quad n=1,2\dots,r,
\]
is such that
\[
T|\mathcal{K}'\sim\underbrace{S(\theta_{N})\oplus S(\theta_{N})\oplus\cdots\oplus S(\theta_{N})}_{r\text{ times}}.
\]
This is seen by noting that $k_{1},k_{2},\dots,k_{r}$ for a $C_{0}$-basis
for $T|\mathcal{K}'$.\end{proof}
\begin{lem}
\label{lem:preliminary to BIG space propo}Fix $m,p,r\in\{1,2,\dots,N\}$
with $r\le p\le N-m$, let $\mathcal{Q}\subset\mathfrak{D}\otimes\mathbb{C}^{N}$
be a subspace of dimension $r$, and let $z_{1},z_{2},\dots,z_{r}\in\mathcal{Q}$
be given by
\[
z_{\ell}=\sum_{n=1}^{m+p}u_{\ell,n}h_{n},\quad\ell=1,2,\dots,r,
\]
where all the coefficients $u_{\ell,n}$ are in $H^{\infty}$. Set
$\mathcal{R}=\mathcal{Q}\cap(H^{2}\otimes\mathbb{C}^{N})$, and $\mathcal{K}=[P_{\mathcal{H}(\Theta)}\mathcal{R}]^{-}$.
Assume that:
\begin{enumerate}
\item $\theta_{m+1}\equiv\theta_{m+2}\equiv\cdots\equiv\theta_{m+p}$ and
$\theta_{m}/\theta_{m+1}$ is not constant.
\item \label{enu:some condition}Every nonconstant inner factor $\omega$
of $\theta_{m+1}$ satisfies $\omega\wedge(\theta_{m}/\theta_{m+1})\not\equiv1$.
\item The Jordan model of $T|\mathcal{K}$ is 
\[
\underbrace{S(\theta_{m+1})\oplus S(\theta_{m+1})\oplus\cdots\oplus S(\theta_{m+1})}_{r\text{ times}}.
\]

\item \label{enu:basis assumption}The vectors $w_{\ell}=P_{\mathcal{H}(\Theta)}z_{\ell}$,
$\ell=1,2,\dots,r$, form a $C_{0}$-basis for $T|\mathcal{K}$.
\item \label{enu:whatever}The set consisting of all inner functions of
the form 
\[
\theta_{n}\wedge\det\begin{bmatrix}u_{1,j_{1}} & u_{1,j_{2}} & \cdots & u_{1,j_{s}}\\
u_{2,j_{1}} & u_{2,j_{2}} & \cdots & u_{2,j_{s}}\\
\vdots & \vdots & \ddots & \vdots\\
u_{s,j_{1}} & u_{s,j_{2}} & \cdots & u_{s,j_{s}}
\end{bmatrix},
\]
where $s\in\{1,2,\dots,r\}$ and $n,j_{1},j_{2},\dots,j_{s}\in\{1,2\dots,N\}$,
is totally ordered by divisibility.
\end{enumerate}

Then there exist integers $m+1\le j_{1}<j_{2}<\cdots<j_{r}\le m+p$
such that
\[
\det[u_{\ell,j_{k}}]_{\ell,k=1}^{r}\wedge\theta_{m+1}\equiv1.
\]

\end{lem}
\begin{proof}
The hypothesis implies that $\theta_{m+1}(T)w_{\ell}=0$ or, equivalently,
\[
\sum_{n=1}^{m+p}\theta_{m+1}(T)u_{\ell,n}(T)h_{n}=0,\quad\ell=1,2,\dots,r.
\]
Since $\{h_{n}\}_{n=1}^{N}$ form a $C_{0}$-basis for $T$, we conclude
that $\theta_{n}|\theta_{m+1}u_{\ell,n}$ for all $n$, in particular
$u_{\ell,n}$ is divisible by $\theta_{m}/\theta_{m+1}$ for $n=1,2,\dots,m$
and $\ell=1,2,\dots,r$. Let $s\in\{0,1,\dots,r\}$ be largest with
the property that there exist 
\[
m+1\le j_{1}<j_{2}<\cdots<j_{s}\le m+p
\]
such that
\[
\det[u_{\ell,j_{m}}]_{\ell,m=1}^{s}\wedge\theta_{m+1}\equiv1,
\]
and assume to get a contradiction that $s<r$. Consider indeterminates
$\xi_{1},\xi_{2},\dots,\xi_{s+1}$ and write the determinant
\[
\det\left[\begin{array}{ccccc}
u_{1,j_{1}} & u_{1,j_{2}} & \cdots & u_{1,j_{s}} & \xi_{1}\\
u_{2,j_{1}} & u_{2,j_{2}} & \cdots & u_{2,j_{s}} & \xi_{2}\\
\vdots & \vdots & \ddots & \vdots & \vdots\\
u_{s,j_{1}} & u_{s,j_{2}} & \cdots & u_{s,j_{s}} & \xi_{s}\\
u_{s+1,j_{1}} & u_{s+1,j_{2}} & \cdots & u_{s+1,j_{s}} & \xi_{s+1}
\end{array}\right]=\sum_{\ell=1}^{s+1}\alpha_{\ell}\xi_{\ell}.
\]
The coefficients $\alpha_{\ell}\in H^{\infty}$ are determinants of
size $s$, and 
\begin{equation}
\alpha_{s+1}\wedge\theta_{m+1}\equiv1.\label{eq:32}
\end{equation}
The vector
\[
z=\sum_{\ell=1}^{s+1}\alpha_{\ell}z_{\ell}=\sum_{n=1}^{m+p}\beta_{n}h_{n}
\]
 has coefficients
\[
\beta_{n}=\sum_{\ell=1}^{s+1}\alpha_{\ell}u_{\ell,n},\quad n=1,2,\dots,N,
\]
which are determinants of order $s+1$. Note that $\beta_{n}=0$ if
$n\in\{j_{1},j_{2},\dots,j_{s}\}$ and $\theta_{m}/\theta_{m+1}$
divides $\beta_{n}$ for $n\in\{1,2,\dots,m\}$. For other values
of $n$, the function $\beta_{n}\wedge\theta_{m+p}$ is not constant
by the definition of $s$. Conditions (\ref{enu:some condition})
and (\ref{enu:whatever}) imply now that there exists a nonconstant
inner function $\omega$ such that $\omega|\beta_{n}\wedge\theta_{m+1}$
for all $n=1,2,\dots,m+p$. Thus the vector $z/\omega$ belongs to
$\mathcal{Q}\cap(H^{2}\otimes\mathbb{C}^{N})$, and therefore $P_{\mathcal{H}(\Theta)}(z/\omega)\in\mathcal{K}$.
This in turn implies that $P_{\mathcal{H}(\Theta)}[(\theta_{m+1}/\omega)z]=\theta_{m+1}(T)P_{\mathcal{H}(\Theta)}(z/\omega)=0$,
so
\[
\frac{\theta_{m+1}}{\omega}z=\sum_{\ell=1}^{s+1}\frac{\theta_{m+1}}{\omega}\alpha_{\ell}z_{\ell}\in M_{\Theta}[H^{2}\otimes\mathbb{C}^{N}],
\]
or
\[
\sum_{\ell=1}^{s+1}(\theta_{m+1}/\omega)(T)\alpha_{\ell}(T)w_{\ell}=0.
\]
Assumption (\ref{enu:basis assumption}) implies that $\theta_{m+1}|(\theta_{m+1}/\omega)\alpha_{\ell}$,
that is, $\omega|\alpha_{\ell}$ for $\ell=1,2,\dots,s+1$. This however
contradicts (\ref{eq:32}) when $\ell=s+1$, thus concluding the proof. 
\end{proof}
Define $\mathcal{E}_{n}$ to be the $\mathfrak{D}$-linear span of
$\{h_{1},h_{2},\dots,h_{n}\}$ when $1\le n\le N$ and $h_{n}\ne0$.
The nonzero vectors $h_{n}$ are linearly independent over $\mathfrak{D}$
because they from a $C_{0}$ basis for $T$, and therefore $\dim_{\mathfrak{D}}\mathcal{E}_{n}=n$.
This sequence of spaces can be supplemented to form a complete flag
$\mathcal{E}$ by defining appropriate spaces $\mathcal{E}_{n}$ of
dimension $n$ when $h_{n}=0$.
\begin{lem}
\label{lem:projection of E_j}Assume that $n\in\{1,2,\dots,$N\} is
such that $\theta_{n}\not\equiv1$. Then 
\[
[P_{\mathcal{H}(\Theta)}(\mathcal{E}_{n}\cap(H^{2}\otimes\mathbb{C}^{N}))]^{-}
\]
 is equal to the smallest invariant subspace for $T$ containing the
vectors $h_{1},h_{2},\dots,h_{n}$.\end{lem}
\begin{proof}
Clearly the space $\mathcal{M}=[P_{\mathcal{H}(\Theta)}(\mathcal{E}_{n}\cap(H^{2}\otimes\mathbb{C}^{N}))]^{-}$
contains the vectors $h_{1},h_{2},\dots,h_{n}$. To conclude the proof,
it suffices to show that these vectors form a $C_{0}$-basis for $T|\mathcal{M}$.
The operator $T|\mathcal{M}$ has multiplicity at most $n$, as seen
at the beginning of the proof of Lemma \ref{lem:First preliminary for large space}.
Thus the Jordan model of $T|\mathcal{M}$ is of the form $\bigoplus_{j=1}^{n}S(\varphi_{j})$
with $\varphi_{j}|\theta_{j}$ for $j=1,2,\dots,n$. On the other
hand, $\theta_{j}|\varphi_{j}$ for $j=1,2,\dots,n$ because $\mathcal{M}$
contains the subset $\{h_{1},h_{2},\dots,h_{n}\}$ of the $C_{0}$-basis
of $T$. The conclusion follows immediately from Proposition \ref{prop:basis when multiplicity is finite}.\end{proof}
\begin{prop}
\label{prop:big-submodule}Fix a positive integer $r\le N$, a subset
\[
I=\{i_{1}<i_{2}<\cdots<i_{r}\}\subset\{1,2,\dots,N\},
\]
and a subspace $\mathcal{Q}\in\mathfrak{S}(\mathcal{E},I)$. Set $\mathcal{R}=\mathcal{Q}\cap(H^{2}\otimes\mathbb{C}^{N})$
and $\mathcal{K}=[P_{\mathcal{H}(\Theta)}\mathcal{R}]^{-}$. Then:
\begin{enumerate}
\item The space $\mathcal{K}$ is invariant for $T=S(\Theta)$, and $T|\mathcal{K}$
has cyclic multiplicity less than or equal to $r$.
\item If the Jordan model of $T|\mathcal{K}$ is $S(\beta_{1})\oplus S(\beta_{2})\oplus\cdots\oplus S(\beta_{r})$,
then $\theta_{i_{x}}$ divides $\beta_{x}$ for $x=1,2,\dots,r$.
\item If $\beta_{x}\equiv\theta_{i_{x}}$ for all $x=1,2,\dots,r$, then
$\mathcal{K}$ has a $T$-invariant quasidirect complement $\mathcal{L}$
in $\mathcal{H}(\Theta)$ such that $T|\mathcal{L}\sim\bigoplus_{i\notin I}S(\theta_{i})$.
\end{enumerate}
\end{prop}
\begin{proof}
We recall that $S\otimes I_{\mathbb{C}^{N}}|\mathcal{R}$ is a unilateral
shift of multiplicity $r$, and therefore it has cyclic multiplicity
$r$ as well. Projecting onto $\mathcal{H}(\Theta)$ a cyclic set
of $r$ elements for $S\otimes I_{\mathbb{C}^{N}}|(\mathcal{Q}\cap(H^{2}\otimes\mathbb{C}^{N}))$
we obtain a cyclic set for $S(\Theta)|\mathcal{R}$, and this yields
(1).

For the proofs of (2) and (3) we need some preparation. For each $x\in\{1,2,\dots,r\}$,
choose a subspace $\mathcal{Q}_{x}\subset\mathcal{Q}\cap\mathcal{E}_{i_{x}}$
such that $\dim_{\mathfrak{D}}\mathcal{Q}_{x}=x$. The subspace $\mathcal{R}_{x}=\mathcal{Q}_{x}\cap(H^{2}\otimes\mathbb{C}^{N})$
is invariant for $S\otimes I_{\mathbb{C}^{N}}$, and $(S\otimes I_{\mathbb{C}^{N}})|\mathcal{R}_{x}$
has multiplicity equal to $x$. Also set $\mathcal{K}_{x}=[P_{\mathcal{H}(\Theta)}\mathcal{R}_{x}]^{-}$.
When $r>1$, we may assume that
\begin{equation}
\mathcal{Q}_{x}\subset\mathcal{Q}_{x+1}\text{, thus }\mathcal{R}_{x}\subset\mathcal{R}_{x+1}\text{ and }\mathcal{K}_{x}\subset\mathcal{K}_{x+1},\quad x=1,2,\dots,r-1.\label{eq:inclusion-of-Q_x}
\end{equation}
To prove (2), it suffices to consider the case when $\theta_{i_{x}}\not\equiv1$.
Denote by 
\[
S(\beta_{1}^{x})\oplus S(\beta_{2}^{x})\oplus\cdots\oplus S(\beta_{x}^{x})
\]
 the Jordan model of $T|\mathcal{K}_{x}$, $x=1,2,\dots,r$. An application
of Lemma \ref{lem:First preliminary for large space} to the space
$\mathcal{Q}_{x}$ shows that $\theta_{i_{x}}|\beta_{x}^{x}$, and
(2) follows because $\mathcal{K}_{x}\subset\mathcal{K}_{r}=\mathcal{K}$
and thus $\beta_{x}^{x}|\beta_{x}^{r}=\beta_{x}$.

As we observed above, we have
\[
\theta_{i_{1}}|\beta_{1}^{x}|\beta_{1},\,\theta_{i_{2}}|\beta_{2}^{x}|\beta_{2},\dots,\,\theta_{i_{x}}|\beta_{x}^{x}|\beta_{x},\quad x=1,2,\dots,r-1.
\]
If the hypothesis of (3) is satisfied, that is, $\beta_{x}\equiv\theta_{i_{x}}$
for $x=1,2,\dots,r$, we must have 
\[
\beta_{j}^{x}\equiv\theta_{i_{x}},\quad j=1,2,\dots,x,
\]
for $x=1,2,\dots,r$. Thus the Jordan model of $T|\mathcal{K}_{x}$
is precisely
\[
S(\theta_{i_{1}})\oplus S(\theta_{i_{2}})\oplus\cdots\oplus S(\theta_{i_{x}})
\]
for $x\in\{1,2,\dots,r\}$. The case in which $\theta_{i_{r}}\equiv1$
reduces to the same statement with $r$ replaced by $r-1$, and therefore
an inductive argument allows us to assume that $\theta_{i_{r}}\not\equiv1$
for the remainder of the proof. A repeated application of Proposition
\ref{prop:completion-of-basis} shows that we can find vectors 
\[
w_{1}\in\mathcal{K}_{1},w_{2}\in\mathcal{K}_{2},\dots,w_{r}\in\mathcal{K}_{r}=\mathcal{K},
\]
such that $\{w_{1},w_{2},\dots,w_{x}\}$ is a $C_{0}$-basis for $T|\mathcal{K}_{x}$,
$x=1,2,\dots,r$. Moreover, we may assume that these vectors are of
the form
\begin{equation}
w_{x}=\sum_{j=1}^{i_{x}}u_{x,j}(T)h_{j},\quad x=1,2,\dots,r,\label{eq:w_x as linear combo}
\end{equation}
where all the $u_{x,j}\in H^{\infty}$. This can be seen as follows.
We have 
\[
\mathcal{K}_{x}\subset[P_{\mathcal{H}(\Theta)}(\mathcal{E}_{i_{x}}\cap(H^{2}\otimes\mathbb{C}^{N}))]^{-},
\]
and according to Lemma \ref{lem:projection of E_j}, the space on
the right is the invariant subspace for $T$ generated by $\{h_{1},h_{2},\dots,h_{i_{x}}\}$.
Denote by $\mathcal{H}_{i}$ the cyclic subspace for $T$ generated
by $h_{i}$, and let 
\[
X:\bigoplus_{j=1}^{i_{x}}\mathcal{H}_{j}\to[P_{\mathcal{H}(\Theta)}(\mathcal{E}_{i_{x}}\cap(H^{2}\otimes\mathbb{C}^{N}))]^{-}
\]
be defined by
\[
X\left[\bigoplus_{j=1}^{i_{x}}v_{j}\right]=\sum_{j=1}^{i_{x}}v_{j},\quad v_{j}\in\mathcal{H}_{j},j=1,2,\dots,i_{x}.
\]
An application of Proposition \ref{prop:dense-vectors} shows that
the vector $w_{x}$ can be chosen to belong to the range of $X$,
thus $w_{x}=\sum_{j=1}^{i_{x}}v_{j}$ for some vectors $v_{j}\in\mathcal{H}_{j}$,
$j=1,2,\dots,i_{x}$. Lemma \ref{lem:vectors in the space of a cyclic operator}
provides functions $\{\alpha_{j},\beta_{j}\}_{j=1}^{i_{x}}\subset H^{\infty}$
such that $\alpha_{j}\wedge\theta_{1}\equiv1$ and $\alpha_{j}(T)v_{j}=\beta_{j}(T)h_{j}$
for $j=1,2,\dots,i_{x}$. The vector $w_{x}$ can then be replaced
by $\alpha(T)w_{x}$, where $\alpha=\alpha_{1}\alpha_{2}\cdots\alpha_{i_{x}}$,
and this vector satisfies (\ref{eq:w_x as linear combo}) with $u_{x,j}=\alpha\beta_{j}$,
$j=1,2,\dots,i_{x}$. For convenience, we write $u_{x,j}=0$ if $j=i_{x}+1,\dots,N$.
Observe also that the condition $\theta_{i_{x}}(T)w_{x}=0$ is equivalent
to
\[
\sum_{j=1}^{i_{x}}\theta_{i_{x}}(T)u_{x,j}(T)h_{j}=0
\]
and the fact that $\{h_{j}\}_{j=1}^{N}$ is a $C_{0}$-basis for $T$
implies that $\theta_{j}|\theta_{i_{x}}u_{x,j}$ for $j=1,2,\dots,i_{x}$.
In other words,
\[
\frac{\theta_{j}}{\theta_{i_{x}}}|u_{x,j},\quad j=1,2,\dots,i_{x},x=1,2,\dots,r.
\]

At this point we need to make a reduction to a special case, namely,
we can assume that the collection consisting of all inner functions
of the form
\[
\theta_{j}\wedge\det\begin{bmatrix}u_{1,j_{1}} & u_{1,j_{2}} & \cdots & u_{1,j_{s}}\\
u_{2,j_{1}} & u_{2,j_{2}} & \cdots & u_{2,j_{s}}\\
\vdots & \vdots & \ddots & \vdots\\
u_{s,j_{1}} & u_{s,j_{2}} & \cdots & u_{s,j_{s}}
\end{bmatrix},
\]
where $1\le s\le r$, $1\le j\le N$, and $1\le j_{1}<j_{2}<\cdots<j_{s}\le N$
and of the inner functions 
\[
\theta_{j}/\theta_{j+1},\quad n=1,2,\dots,N-1,
\]
 is totally ordered by divisibility and, for every nonconstant inner
factor $\omega$ of $\theta_{1}$ we have $\omega\wedge\theta_{j}\not\equiv1$,
$j=1,2,...,N\lyxmathsym{\textminus}1$, unless $\theta_{j}$ itself
is constant. This is accomplished as follows. Use Lemma \ref{lem:achieving-total order}
to find a decomposition of $\theta_{1}=m_{T}$ into a product
\[
\theta_{1}=\gamma_{1}\gamma_{2}\cdots\gamma_{n}
\]
of relatively prime inner factors with the property that the collection
consisting of all  inner functions of the form
\[
\gamma_{\ell}\wedge\theta_{j}\wedge\det\begin{bmatrix}u_{1,j_{1}} & u_{1,j_{2}} & \cdots & u_{1,j_{s}}\\
u_{2,j_{1}} & u_{2,j_{2}} & \cdots & u_{2,j_{s}}\\
\vdots & \vdots & \ddots & \vdots\\
u_{s,j_{1}} & u_{s,j_{2}} & \cdots & u_{s,j_{s}}
\end{bmatrix},
\]
where $1\le s\le r$, $1\le j\le N$, and $1\le j_{1}<j_{2}<\cdots<j_{s}\le N$
and of the inner functions 
\[
\gamma_{\ell}\wedge(\theta_{j}/\theta_{j+1})\quad,j=1,2,\dots,N-1,
\]
 is totally ordered by divisibility, and such that condition (2) of
that lemma is satisfied as well. Setting $\Gamma_{\ell}=\theta_{1}/\gamma_{\ell}$,
Proposition \ref{prop:primary-decomposition} shows that it suffices
to show that $[\Gamma_{\ell}(T)\mathcal{R}]^{-}$ has an invariant
quasidirect complement in $[\Gamma_{\ell}(T)\mathcal{H}(\Theta)]^{-}$
for $\ell=1,2,\dots,n$. In order to do this, we replace $T|[\Gamma_{\ell}(T)\mathcal{H}(\Theta)]^{-}$
by a quasisimilar operator as follows. Define  $A_{\ell}:H^{2}\otimes\mathbb{C}^{N}\to[\Gamma_{\ell}(T)\mathcal{H}(\Theta)]^{-}$
by setting
\[
A_{\ell}h=P_{\mathcal{H}(\Theta)}(\Gamma_{\ell}h).
\]
We have $TA_{\ell}=A_{\ell}(S\otimes I_{\mathbb{C}^{N}})$, so that
the injective operator $B_{\ell}|[(H^{2}\otimes\mathbb{C}^{N})\ominus\ker(A_{\ell})]$
satisfies
\[
TB_{\ell}=B_{\ell}T_{\ell},
\]
where $T_{\ell}$ is the compression of $S\otimes I_{\mathbb{C}^{N}}$
to $\mathcal{H}_{\ell}=(H^{2}\otimes\mathbb{C}^{N})\ominus\ker(A_{\ell})$.
It follows immediately that $T{}_{\ell}$ is an operator of class
$C_{0}$ and, since the range of $B_{\ell}$ is dense in $[\Gamma_{\ell}(T)\mathcal{H}(\Theta)]^{-}$,
$T_{\ell}$ is quasisimilar to $T|[\Gamma_{\ell}(T)\mathcal{H}(\Theta)]^{-}$.
The vectors $P_{\mathcal{H}_{\ell}}h_{1},P_{\mathcal{H}_{\ell}}h_{2},\dots,P_{\mathcal{H}_{\ell}}h_{N}$
form a $C_{0}$-basis for $T_{\ell}$ because the vectors 
\[
B_{\ell}P_{\mathcal{H}_{\ell}}h_{j}=\Gamma_{\ell}(T)h_{j}
\]
 form a $C_{0}$-basis for $T|[\Gamma_{\ell}(T)\mathcal{H}(\Theta)]^{-}$.
Define next a subspace $\mathcal{K}^{\ell}\subset\mathcal{H}_{\ell}$
as the invariant subspace for $T_{\ell}$ generated by $P_{\mathcal{H}_{\ell}}w_{x}$,
$x=1,2,\dots,r$. Then $[B_{\ell}\mathcal{K}^{\ell}]^{-}=[\Gamma_{\ell}(T)\mathcal{K}]^{-}$
and the vectors $P_{\mathcal{H}_{\ell}}w_{x}$, $x=1,2,\dots,r$ form
a $C_{0}$-basis for $T_{\ell}|\mathcal{K}^{\ell}$. We see that it
suffices to show that $\mathcal{K}^{\ell}$ has an invariant quasi-complement
in $\mathcal{H}_{\ell}$. The spaces $\mathcal{K}^{\ell}$ and $\mathcal{H}_{\ell}$
are entirely analogous to the original spaces $\mathcal{K}$ and $\mathcal{H}$,
and they have the additional divisibility properties outlined above.
Of course, $\theta_{i}\wedge\gamma_{\ell}\equiv\theta_{i}^{T_{\ell}}$.
Thus it suffices to prove (3) under these additional divisibility
conditions. This is accomplished by showing that the vectors $\{w_{x}\}_{x=1}^{r}$
and a set $F\subset\{h_{j}\}_{j=1}^{N}$ of cardinality $N-r$ form
a $C_{0}$-basis for $T$. The quasidirect complement of $\mathcal{K}$
is then the invariant subspace for $T$ generated by $F$. Divide
the functions $\{\theta_{i_{x}}\}_{x=1}^{r}$ into equivalence classes,
that is, find integers $0=r_{0}<r_{1}<r_{2}<\cdots<r_{p}=r$ such
that $\theta_{i_{x}}\equiv\theta_{i_{x+1}}$ when $r_{k-1}<x\le r_{k}$
and $1\le i\le p$, but $\theta_{i_{x+1}}/\theta_{i_{x}}$ is not
constant if $x=r_{k}$, $1\le k<p$. We apply Lemma \ref{lem:preliminary to BIG space propo}
to the $\mathfrak{D}$-vector space generated by $\{w_{x}\}_{x=r_{k-1}+1}^{r_{k}}$
to obtain indices $\{j_{m}\}_{m=r_{k-1}+1}^{r_{k}}$ with the property
that $\theta_{j_{m}}\equiv\theta_{i_{r_{k}}}$ for $r_{k-1}+1\le m\le r_{k}$,
and the function
\[
\det[u_{\ell,j_{k}}]_{\ell,k=r_{i-1}+1}^{r_{i}}
\]
is relatively prime to $\theta_{i_{r_{k}}}$, and therefore to $\theta_{1}$
because of the divisibility assumptions we made about $\theta_{1},\theta_{2},\dots,\theta_{N}$.
The proof of (3) is now completed by showing that the set
\[
A=\{w_{\ell}\}_{\ell=1}^{r}\cup\{h_{j}:j\notin\{j_{1},j_{2},\dots,j_{r}\}\}
\]
is a $C_{0}$-basis for $T$. In other words, this basis is obtained
by replacing each $h_{j_{\ell}}$ by the corresponding vector $w_{\ell}$.
It is clear however that the set $A$ is obtained from the $C_{0}$-basis
$\{h_{j}\}_{j=1}^{N}$ by the process described in Corollary \ref{cor:base change}.
This concludes the proof of (3) and of the proposition.
\end{proof}
The preceding proposition shows that, for certain flags $\mathcal{E}$
and for spaces $\mathcal{Q}\in\mathfrak{S}(\mathcal{E},I)$, the invariant
subspace $\mathcal{M}=[P_{\mathcal{H}(\Theta)}(\mathcal{Q}\cap(H^{2}\otimes\mathbb{C}^{N}))]^{-}$
has the property that $T|\mathcal{M}$ has a `large' Jordan model.
Next we produce flags $\mathcal{E}$ with the property that the Jordan
models of operators of the form $T|\mathcal{M}$ are `small' if $\mathcal{Q}\in\mathfrak{S}(\mathcal{E},I)$.
Some preparation is needed first.
\begin{lem}
\label{lem:projections and invariant subspaces}Consider a factorization
\[
\Theta(\lambda)=\Theta''(\lambda)\Theta'(\lambda),\quad\lambda\in\mathbb{D},
\]
 where $\Theta'$ and $\Theta''$ are inner functions. Write 
\[
\mathcal{H}(\Theta)=\mathcal{H}'\oplus\mathcal{H}'',
\]
where 
\[
\mathcal{H}'=M_{\Theta''}[H^{2}\otimes\mathbb{C}^{N}]\ominus M_{\Theta}[H^{2}\otimes\mathbb{C}^{N}]=M_{\Theta''}\mathcal{H}(\Theta')
\]
is $T$-invariant and $\mathcal{H}''=\mathcal{H}(\Theta'')$. Let
$\mathcal{Q}\subset\mathfrak{D}\otimes\mathbb{C}^{N}$ be a $\mathfrak{D}$-vector
space, and define invariant subspaces
\begin{eqnarray*}
\mathcal{M} & = & [P_{\mathcal{H}(\Theta)}(\mathcal{Q}\cap(H^{2}\otimes\mathbb{C}^{N}))]^{-},\\
\mathcal{M}' & = & [P_{\mathcal{H}'}(\mathcal{Q}\cap M_{\Theta''}[H^{2}\otimes\mathbb{C}^{N}])]^{-},\\
\mathcal{M}'' & = & [P_{\mathcal{H}''}(\mathcal{Q}\cap(H^{2}\otimes\mathbb{C}^{N}))]^{-},
\end{eqnarray*}
for $T$, $T'=T|\mathcal{H}'$, and $T''=P_{\mathcal{H}''}T|\mathcal{H}''$,
respectively. Then we have
\[
\mathcal{M}''=[P_{\mathcal{H}''}\mathcal{M}]^{-}\text{ and }\mathcal{M}'=\mathcal{M}\cap\mathcal{H}'.
\]
\end{lem}
\begin{proof}
The first equality and the inclusion
\[
\mathcal{M}'\subset\mathcal{M}\cap\mathcal{H}'
\]
follow directly from the definitions of $\mathcal{M},\mathcal{M}',$
and $\mathcal{M}''$. For the opposite inclusion, note that the set
of vectors $h'\in\mathcal{M}\cap\mathcal{H}'$ of the form $P_{\mathcal{H}(\Theta)}q$
for some $q\in\mathcal{Q}\cap(H^{2}\otimes\mathbb{C}^{N})$ is dense
in $\mathcal{M}\cap\mathcal{H}'$ by Proposition \ref{prop:dense-vectors}.
Consider such a vector $h'=P_{\mathcal{H}(\Theta)}q$, and observe
that $q-h'\in M_{\Theta}[H^{2}\otimes\mathbb{C}^{N}]$ so that
\[
q\in\mathcal{H}'+M_{\Theta}[H^{2}\otimes\mathbb{C}^{N}]\subset M_{\Theta''}[H^{2}\otimes\mathbb{C}^{N}].
\]
It follows that $h'\in\mathcal{M}'$, and this concludes the proof.
\end{proof}
The preceding lemma is applied in the proofs of Proposition \ref{prop:small invariant subspace}
and Theorem \ref{thm:Horn inequalities-finite case}. For the first
application, we recall that any inner multiple $\omega$ of $\theta_{1}$
is a scalar multiple of $\Theta$, that is, there exists an inner
function $\Omega$ such that
\begin{equation}
\Theta(\lambda)\Omega(\lambda)=\omega(\lambda)I_{\mathbb{C}^{N}},\quad\lambda\in\mathbb{D}.\label{eq:scalar multiple}
\end{equation}
 Of course, the operator $S(\omega I_{\mathbb{C}^{N}})$ is unitarily
equivalent to the orthogonal sum of $N$ copies of $S(\omega)$.
\begin{lem}
\label{lem:C0 basis in space and subspace}Let $\omega\in H^{\infty}$
be an inner multiple of $\theta_{1}$\emph{.} Then there exist bounded
vectors $y_{1},y_{2},\dots,y_{N}\in H^{2}\otimes\mathbb{C}^{N}$ such
that:
\begin{enumerate}
\item $P_{\mathcal{H}(\omega I_{\mathbb{C}^{N}})}y_{1},P_{\mathcal{H}(\omega I_{\mathbb{C}^{N}})}y_{2},\dots,P_{\mathcal{H}(\omega I_{\mathbb{C}^{N}})}y_{N}$
form a $C_{0}$-basis for $S(\omega)\otimes I_{\mathbb{C}^{N}}$,
\item $P_{\mathcal{H}(\Theta)}y_{1},P_{\mathcal{H}(\Theta)}y_{2},\dots,P_{\mathcal{H}(\Theta)}y_{N}$
form a $C_{0}$-basis for $T=S(\Theta)$, and
\item $P_{\mathcal{H}(\omega I_{\mathbb{C}^{N}})}(\theta_{N}y_{N}),P_{\mathcal{H}(\omega I_{\mathbb{C}^{N}})}(\theta_{N-1}y_{N-1}),\dots,P_{\mathcal{H}(\omega I_{\mathbb{C}^{N}})}(\theta_{1}y_{1})$
form a $C_{0}$-basis for the restriction of $S\otimes I_{\mathbb{C}^{N}}$
to $M_{\Theta}[H^{2}\otimes\mathbb{C}^{N}]\ominus[\omega H^{2}\otimes\mathbb{C}^{N}]$.
\end{enumerate}
\end{lem}
\begin{proof}
Denote by $\Theta'$ the diagonal matrix with diagonal entries $\theta_{1},\dots,\theta_{N}$.
As noted earlier, $\Theta$ and $\Theta'$ are quasiequivalent. Choose
$N\times N$ matrices $X,Y$ over $H^{\infty}$ such that $\Theta X=Y\Theta'$
and $\det(X)\wedge\omega\equiv\det(Y)\wedge\omega\equiv1$, and let
$y_{j}=Y(1\otimes e_{j})$, $j=1,\dots,N$, be the columns of $Y$.
We claim that these vectors satisfy the conclusion of the lemma. Indeed,
(2) follows from Lemma \ref{lem:troubling-lemma-1} while (1) follows
from the same lemma because $(\omega\otimes I_{\mathbb{C}^{N}})Y=Y(\omega\otimes I_{\mathbb{C}^{N}})$.
Finally, we observe that the vectors $P_{\mathcal{H}(\omega I_{\mathbb{C}^{N}})}(\theta_{j}y_{j})$
belong to $M_{\Theta}[H^{2}\otimes\mathbb{C}^{N}]\ominus[\omega H^{2}\otimes\mathbb{C}^{N}]$
because $\theta_{j}y_{j}=\Theta X(1\otimes e_{j})$ for $j=1,\dots,N$.
Moreover, these vectors form a $C_{0}$-basis for the $S(\omega)\otimes I_{\mathbb{C}^{N}}$-invariant
subspace $\mathcal{M}$ they generate. To conclude the proof of (3)
we need to show that $\mathcal{M}=M_{\Theta}[H^{2}\otimes\mathbb{C}^{N}]\ominus[\omega H^{2}\otimes\mathbb{C}^{N}]$
and for this purpose it suffices to verify that the restrictions of
$S(\omega)\otimes I_{\mathbb{C}^{N}}$ to these two subspaces have
the same Jordan model. This follows from the main result of \cite{ber-tan}.
\end{proof}
For our result on `small' invariant subspaces, we fix an inner multiple
$\omega\in H^{\infty}$ of $\theta_{1}^{2}$, a sequence $\{y_{n}\}_{n=1}^{N}$
satisfying the conclusion of Lemma \ref{lem:C0 basis in space and subspace},
and denote by $\mathcal{F}$ the complete flag in $\mathfrak{D}\otimes\mathbb{C}^{N}$
defined by letting $\mathcal{F}_{n}$ be the space generated by $\{y_{N},y_{N-1},\dots,y_{N-n+1}\}$
for $n=1,2,\dots,N$. The choice of $\omega$ insures that all the
functions $\omega/\theta_{n}$ are divisible by $\theta_{1}$.
\begin{prop}
\label{prop:small invariant subspace}Fix a positive integer $r\le N$,
a subset 
\[
I=\{i_{1}<i_{2}<\cdots<i_{r}\}\subset\{1,2,\dots,N\},
\]
and a subspace $\mathcal{Q}\in\mathfrak{S}(\mathcal{F},I)$. Set $\mathcal{R}=\mathcal{Q}\cap(H^{2}\otimes\mathbb{C}^{N})$
and $\mathcal{K}=[P_{\mathcal{H}(\Theta)}\mathcal{R}]^{-}$. Then\emph{:}
\begin{enumerate}
\item The space $\mathcal{K}$ is invariant for $T=S(\Theta)$, and $T|\mathcal{K}$
has cyclic multiplicity less than or equal to $r$.
\item If the Jordan model of $T|\mathcal{K}$ is $S(\alpha_{1})\oplus S(\alpha_{2})\oplus\cdots\oplus S(\alpha_{r})$,
then $\alpha_{x}$ divides $\theta_{N+1-i_{r+1-x}}$ for $x=1,2,\dots,r$.
\item If $\alpha_{x}\equiv\theta_{N+1-i_{r+1-x}}$ for all $x=1,2,\dots,r$,
then $\mathcal{K}$ has a $T$-invariant quasidirect complement $\mathcal{L}$
in $\mathcal{H}(\Theta)$ such that $T|\mathcal{L}\sim\bigoplus_{N+1-i\notin I}S(\theta_{i})$.
\end{enumerate}
\end{prop}
\begin{proof}
Using (\ref{eq:scalar multiple}) and the notation in Lemma \ref{lem:projections and invariant subspaces},
with $\omega I_{\mathbb{C}^{N}},\Omega,\Theta$ playing the roles
of $\Theta,\Theta',\Theta''$, respectively, we have $\mathcal{K}=\mathcal{M}''$.
Continuing with that notation, we claim that the compression of $S\otimes I_{\mathbb{C}^{N}}$
to the space $\mathcal{M}$ has Jordan model 
\[
\underbrace{S(\omega)\oplus S(\omega)\oplus\cdots\oplus S(\omega)}_{r\text{ times}}.
\]
Indeed, apply Proposition \ref{prop:big-submodule} with $S(\omega)\otimes I_{\mathbb{C}^{N}}$
in place of $T$ to deduce that this Jordan model has the form $S(\gamma_{1})\oplus S(\gamma_{2})\oplus\cdots\oplus S(\gamma_{r})$
where all the functions $\gamma_{j}$ are divisible by $\omega$.
However, these functions must also divide the minimal function of
$S(\omega)\otimes I_{\mathbb{C}^{N}}$ which is $\omega$, and therefore
$\gamma_{j}\equiv\omega$ for $j=1,2,\dots,r$.

Next we observe that
\[
T'=S(\omega)\otimes I_{\mathbb{C}^{N}}|[\mathcal{H}(\omega)\otimes\mathbb{C}^{N}]\ominus\mathcal{H}(\Theta)
\]
is an operator of class $C_{0}$ with Jordan model 
\[
S(\omega/\theta_{N})\oplus S(\omega/\theta_{N-1})\oplus\cdots\oplus S(\omega/\theta_{1}),
\]
where the summands are written in this order so that this is a Jordan
operator. Indeed, this follows from \cite{ber-tan}. We can now apply
Proposition \ref{prop:big-submodule} with $T'$ and $(S\otimes I_{\mathbb{C}^{N}})|M_{\Theta}[H^{2}\otimes\mathbb{C}^{N}]$
in place of $T$ and $S\otimes I_{\mathbb{C}^{N}}$ to deduce that
the Jordan model of $T'|\mathcal{M}'$ is of the form $S(\beta_{1})\oplus S(\beta_{2})\oplus\cdots\oplus S(\beta_{r})$
with the property that $\omega/\theta_{N+1-i_{x}}$ divides $\beta_{x}$
for $x=1,2,\dots,r$. It follows from \cite{ber-tan} that the Jordan
model of the compression of $T$ to $\mathcal{M}\ominus\mathcal{M}'$
has Jordan model $S(\alpha_{1})\oplus S(\alpha_{2})\oplus\cdots\oplus S(\alpha_{r})$,
where $\alpha_{r+1-x}\beta_{x}=\omega$ for $x=1,2,\dots,r$. We deduce
that $\alpha_{r+1-x}=\omega/\beta_{x}$ divides
\[
\frac{\omega}{\omega/\theta_{N+1-i_{x}}}=\theta_{N+1-i_{x}}
\]
or, equivalently, $\alpha_{x}|\theta_{N+1-i_{r+1-x}}$ for $x=1,2,\dots,r$.
Parts (1) and (2) of the statement follow now once we prove that $T|\mathcal{M}''=T|\mathcal{K}$
is quasisimilar to the compression of $T$ to $\mathcal{M}\ominus\mathcal{M}'$.
In fact, an operator $X$ which is one-to-one with dense range and
intertwines these two operators is obtained by setting $X=P_{\mathcal{H}(\Theta)}|\mathcal{M}\ominus\mathcal{M}'$.
The claimed properties of $X$ follow readily from Lemma \ref{lem:projections and invariant subspaces}.

Assume finally that $\alpha_{x}\equiv\theta_{N+1-i_{r+1-x}}$ for
$x=1,2,\dots,r$. We have then $\beta_{x}\equiv\omega/\theta_{N+1-x}$
for $x=1,2,\dots,r.$ The proof of Proposition \ref{prop:big-submodule}(3)
can now be applied with $T'$ in place of $T$ and $M_{\Theta}[H^{2}\otimes\mathbb{C}^{N}]$
in place of $H^{2}\otimes\mathbb{C}^{N}$. Following that argument,
and recalling that $P_{\mathcal{H}'}(\theta_{N+1-j}y_{N+1-j})$ form
a $C_{0}$-basis for $T'$, we first produce a $C_{0}$-basis $w_{1},w_{2},\dots,w_{r}$
for $T'|\mathcal{M}'$ such that
\[
w_{x}=P_{\mathcal{H}'}\sum_{j=1}^{N}u_{x,j}\theta_{j}y_{j},\quad x=1,2,\dots,r,
\]
with $\sum_{j=1}^{N}u_{x,j}\theta_{j}z_{j}\in\mathcal{Q}$, coefficients
$u_{x,j}\in H^{\infty}$ such that $u_{x,j}=0$ for $j=1,2,\dots,N+1-i_{x}$
and $\theta_{N+1-i_{x}}/\theta_{j}$ divides $u_{x,j}$ for $j\ge N-i_{x}$.
We now apply a reduction which allows us to assume that a family of
inner functions is totally ordered by divisibility. We partition $I$
into subsets $I_{1},I_{2},\dots,I_{p}$ with the property that $\theta_{N+1-i}\equiv\theta_{N+1-i'}$
if and only if $i$ and $i'$ belong to the same $I_{k}$, and apply
the arguments in the proof of Proposition \ref{prop:big-submodule}
to find minors of the form
\[
\det[u_{x,j}]_{N+1-x\in I_{k},N+1-j\in J_{k}}
\]
which are relatively prime to $\theta_{1}$. Here $J_{k}$ and $I_{k}$
have the same cardinality, and $\theta_{N+1-j}\equiv\theta_{N+1-i}$
for $i\in I_{k}$ and $j\in J_{k}$. The functions
\[
v_{x,j}=\frac{u_{x,j}\theta_{j}}{\theta_{N+1-i_{x}}}
\]
belong again to $H^{\infty}$, and thus we can define vectors $\widetilde{w}{}_{1},\widetilde{w}{}_{2},\dots,\widetilde{w}{}_{r}\in\mathcal{H}(\Theta)$
by setting
\[
\widetilde{w}{}_{x}=P_{\mathcal{H}(\Theta)}\sum_{j=1}^{N}v_{x,j}y_{j},\quad x=1,2,\dots,r.
\]
Since $\sum_{j=1}^{N}v_{x,j}y_{j}\in\mathcal{Q}$, these vectors actually
belong to $\mathcal{M}''=\mathcal{K}$, and they form a $C_{0}$-basis
for $T|\mathcal{K}$. The argument is now concluded by observing that
the set 
\[
\{\widetilde{w}{}_{x}\}_{x=1}^{r}\cup\left\{ P_{\mathcal{H}(\Theta)}y_{j}:j\in\{1,2,\dots,N\}\setminus\bigcup_{q=1}^{p}J_{q}\right\} 
\]
is a $C_{0}$-basis for $T$, so that an invariant quasidirect complement
for $\mathcal{K}$ is generated by the vectors
\[
\left\{ P_{\mathcal{H}(\Theta)}y_{j}:j\in\{1,2,\dots,N\}\setminus\bigcup_{q=1}^{p}J_{q}\right\} .\qedhere
\]

\end{proof}

\section{The Horn inequalities\label{sec:The-Horn-inequalities}}

In this section we consider an integer $N$ and three sets $I,J,K\subset\{1,2,\dots,N\}$,
each containing $r\le N$ elements, such that $c_{I\widetilde{J}\widetilde{K}}=1$,
where
\[
\widetilde{J}=\{N+1-j:j\in J\},\quad\widetilde{K}=\{N+1-k:k\in K\}.
\]
The Horn inequalities associated with such triples of sets are sufficient
to imply all the Horn inequalities associated with sets $I,J,K$ such
that $c_{I\widetilde{J}\widetilde{K}}>0$ (see \cite{belk} or \cite{KTW}).
Our main result is as follows.
\begin{thm}
\label{thm:Horn inequalities-finite case}Assume that $T$ is an operator
of class $C_{0}$ on $\mathcal{H}$, $\mathcal{H}'$ is an invariant
subspace for $T$, $\mathcal{H}''=\mathcal{H}\ominus\mathcal{H}'$,
and
\[
T=\begin{bmatrix}T' & *\\
0 & T''
\end{bmatrix}
\]
is the matrix of $T$ corresponding to the orthogonal decomposition
$\mathcal{H}=\mathcal{H}'\oplus\mathcal{H}''$. Let
\[
\bigoplus_{1\le n<\aleph}S(\theta_{n}),\ \bigoplus_{1\le n<\aleph}S(\theta'_{n}),\ \bigoplus_{1\le n<\aleph}S(\theta_{n}^{\prime\prime})
\]
 be the Jordan models of $T,T'$, and $T''$, respectively. Let $I,J,K\subset\{1,2,\dots,N\}$
be sets of cardinality $r\le N$ such that $c_{I\widetilde{J}\widetilde{K}}=1$.
Then there exists an invariant subspace $\mathcal{M}$ for $T$ with
the following properties.
\begin{enumerate}
\item The cyclic mutiplicity of $T|\mathcal{M}$ is at most $r$, and its
Jordan model $\bigoplus_{x=1}^{r}S(\beta_{x})$ is such that $\theta_{i_{x}}|\beta_{x}$
for $x=1,2,\dots,r$.
\item The Jordan model $\bigoplus_{x=1}^{r}S(\beta'_{x})$ of $T|\mathcal{M}'$,
where $\mathcal{M}'=\mathcal{M}\cap\mathcal{\mathcal{H}}'$ is such
that $\beta'_{x}|\theta'_{j_{x}}$ for $x=1,2,\dots,r$.
\item The Jordan model $\bigoplus_{x=1}^{r}S(\beta_{x}^{\prime\prime})$
of $T|\mathcal{M}''$, where $\mathcal{M}''=\overline{P_{\mathcal{H}''}\mathcal{M}}$,
is such that $\beta_{x}^{\prime\prime}|\theta_{k_{x}}^{\prime\prime}$
for $x=1,2,\dots,r$.
\item $\prod_{x=1}^{r}\beta_{x}=\prod_{x=1}^{r}(\beta'_{x}\beta_{x}^{\prime\prime})$.
\end{enumerate}

We conclude that
\[
\prod_{i\in I}\theta_{i}\left|\prod_{j\in J}\theta'_{j}\prod_{k\in K}\theta_{k}^{\prime\prime}.\right.
\]
If $\prod_{i\in I}\theta_{i}\equiv\prod_{j\in J}\theta'_{j}\prod_{k\in K}\theta_{k}^{\prime\prime}$
then $\mathcal{M}$ \emph{(}respectively, $\mathcal{M}',\mathcal{M}''$\emph{)
}has a $T$-invariant \emph{(}respectively, $T'$-invariant, $T''$-invariant\emph{)}
quasidirect complement in $\mathcal{H}$ \emph{(}respectively, $\mathcal{H},\mathcal{H}''$\emph{)}.

\end{thm}
\begin{proof}
The key case to consider is that in which $T$ has cyclic multiplicity
at most equal to $N$. In this case we can replace $T$ by any operator
of class $C_{0}$ which is quasisimilar to it. Indeed, quasisimilarity
between operators of class $C_{0}$ with finite multiplicity allows
one to identify their lattices of invariant subspaces (see \cite[Proposition VII.1.21]{C0-book}).
We can then assume that $T=S(\Theta)$, where $\Theta$ is an $N\times N$
inner function, $\mathcal{H}=\mathcal{H}(\Theta)$, $\mathcal{H}''=\mathcal{H}(\Theta'')$,
and $\mathcal{H}'=M_{\Theta''}\mathcal{H}(\Theta')$ for some inner
factorization $\Theta=\Theta''\Theta'$. Propositions \ref{prop:big-submodule}
and \ref{prop:small invariant subspace} allow us to choose three
complete flags $\mathcal{E},\mathcal{F},\mathcal{G}$ in $\mathfrak{D}^{N}=\mathfrak{D}\otimes\mathbb{C}^{N}$
such that, given a space $\mathcal{Q}\in\mathfrak{S}(\mathcal{E},I)$
(respectively, $\mathfrak{S}(\mathcal{F},\widetilde{J}),\mathfrak{S}(\mathcal{G},\widetilde{K})$)
the restriction of $T$ (respectively, $T',T''$) to $[P_{\mathcal{H}}(\mathcal{Q}\cap(H^{2}\otimes\mathbb{C}^{N}))]^{-}$
(respectively, $[P_{\mathcal{H}'}(\mathcal{Q}\cap M_{\Theta''}(H^{2}\otimes\mathbb{C}^{N}))]^{-}$,
$[P_{\mathcal{H}''}(\mathcal{Q}\cap(H^{2}\otimes\mathbb{C}^{N}))]^{-}$)
has Jordan model $\bigoplus_{x=1}^{r}S(\beta_{x})$ (respectively,
$\bigoplus_{x=1}^{r}S(\beta'_{x})$, $\bigoplus_{x=1}^{r}S(\beta_{x}^{\prime\prime})$)
satisfying $\theta_{i_{x}}|\beta_{x}$ (respectively, $\beta'_{x}|\theta'_{j_{x}}$,
$\beta_{x}^{\prime\prime}|\theta_{k_{x}}^{\prime\prime}$) for $x=1,2,\dots,r$.
The assumption that $c_{I\widetilde{J}\widetilde{K}}=1$ implies now
that we can find a space $\mathcal{Q}$ in the intersection
\[
\mathfrak{S}(\mathcal{E},I)\cap\mathfrak{S}(\mathcal{F},\widetilde{J})\cap S(\mathcal{G},\widetilde{K}),
\]
and this yields immediately statements (1), (2), and (3). Statement
(4) follows from \cite[Theorem VI.3.16]{C0-book} and the fact that
$T''|\mathcal{M}''$ is quasisimilar to the compression of $T|\mathcal{M}$
to the space $\mathcal{M}\ominus\mathcal{M}'$. If $\prod_{i\in I}\theta_{i}\equiv\prod_{j\in J}\theta'_{j}\prod_{k\in K}\theta_{k}^{\prime\prime}$
then, of course, $\beta_{x}\equiv\theta_{i_{x}}$, $\beta'_{x}\equiv\theta'_{j_{x}}$,
and $\beta_{x}^{\prime\prime}\equiv\theta_{x}^{\prime\prime}$ for
$x=1,2,\dots,r$, and the existence of invariant quasidirect complements
follows from part (3) of Propositions \ref{prop:big-submodule} and
\ref{prop:small invariant subspace}.

We consider next an operator $T$ of class $C_{0}$ with finite multiplicity
$N'>N$. This case reduces to the previous one as follows. The sets
$I,J,K$ are also contained in $\{1,2,\dots,N'\}$, and setting
\[
\overline{J}=\{N'+1-j:j\in J\},\quad\overline{K}=\{N'+1-k:k\in K\},
\]
we still have $c_{I\overline{J}\,\overline{K}}=1$. (This fact is
verified using the Littlewood-Richardson rule for partitions. More
generally, $c_{I\overline{J}\,\overline{K}}=c_{I\widetilde{J}\widetilde{K}}$,
see \cite{young-tab,fulton-survey}.) Therefore the preceding argument
works simply replacing $N$ by $N'$.

Finally, assume that $T$ has infinite multiplicity, and consider
quasidirect decompositions
\[
\mathcal{H}=\bigvee_{1\le n<\aleph}\mathcal{H}_{n},\quad\mathcal{H}'=\bigvee_{1\le n<\aleph}\mathcal{H}'_{n},\quad\mathcal{H}''=\bigvee_{1\le n<\aleph}\mathcal{H}_{n}^{\prime\prime},
\]
 into invariant spaces for $T,T',T'',$ respectively, such that $T|\mathcal{H}_{n}\sim S(\theta_{n})$,
$T'|\mathcal{H}'_{n}\sim S(\theta_{n}')$, and $T''|\mathcal{H}_{n}^{\prime\prime}\sim S(\theta_{n}^{\prime\prime})$
for all $n<\aleph$. The invariant subspace $\widetilde{\mathcal{H}}$
for $T$ generated by the spaces $\mathcal{H}_{n},\mathcal{H}'_{n},\mathcal{H}_{n}^{\prime\prime}$
for $1\le n\le N$ has the property that $T|\widetilde{\mathcal{H}}$
has finite multiplicity. We can therefore apply the theorem, already
proved for the case of operators with finite multiplicity, with $T|\widetilde{\mathcal{H}}$
in place of $T$, $\widetilde{\mathcal{H}}'=\mathcal{H}'\cap\widetilde{\mathcal{H}}$
in place of $\mathcal{H}'$, and $\widetilde{\mathcal{H}}''=\widetilde{\mathcal{H}}\ominus\widetilde{\mathcal{H}}'$
in place of $\mathcal{H}''$. We obtain a $T|\widetilde{\mathcal{H}}$-invariant
subspace $\mathcal{M}\subset\widetilde{\mathcal{H}}$ which satisfies
requirements (1--4) of the theorem. Indeed, the first $N$ functions
in the Jordan model of $T|\widetilde{\mathcal{H}}$ are still $\theta_{1},\theta_{2},\dots,\theta_{N}$,
and similar observations hold for $T|\widetilde{\mathcal{H}}'$ and
the compression of $T$ to $\widetilde{\mathcal{H}}''$. It remains
to verify the final assertion of the theorem. Assume therefore that
$\prod_{i\in I}\theta_{i}\equiv\prod_{j\in J}\theta'_{j}\prod_{k\in K}\theta_{k}^{\prime\prime}$.
We already know that $\mathcal{M}$ has a $T|\widetilde{\mathcal{H}}$-invariant
quasidirect complement in $\widetilde{\mathcal{H}}$. Setting $N'=\mu_{T|\widetilde{\mathcal{H}}}$,
there exists a quasidirect decomposition
\[
\widetilde{\mathcal{H}}=\bigvee_{n=1}^{N'}\widetilde{\mathcal{H}}_{n}
\]
into invariant subspaces for $T$ such that $T|\widetilde{\mathcal{H}}_{n}\sim S(\theta_{n})$
for $n=1,2,\dots,N$, and
\[
\mathcal{M}=\bigvee_{i\in I}\widetilde{\mathcal{H}}_{i}.
\]
Proposition \ref{prop:completion-of-basis} allows us to construct
an invariant subspace $\mathcal{L}$ for $T$ such that
\[
\mathcal{H}=\mathcal{L}\vee\bigvee_{n=1}^{N}\widetilde{\mathcal{H}}_{n},\quad\mathcal{L}\cap\left[\bigvee_{n=1}^{N}\widetilde{\mathcal{H}}_{n}\right]=\{0\}.
\]
An invariant quasidirect complement for $\mathcal{M}$ is then given
by
\[
\mathcal{L}\vee\bigvee_{n\not\in I}\widetilde{\mathcal{H}}_{n}.
\]
Similar arguments show the existence of invariant quasidirect complements
for $\mathcal{M}'$ and $\mathcal{M}''$.
\end{proof}

\section{Operators on nonseparable spaces\label{sec:Operators-on-nonseparable}}

With the notation of Theorem \ref{thm:Horn inequalities-finite case},
we show that $\theta_{\beta}|\theta'_{\beta}\theta_{\beta}^{\prime\prime}$
for $\beta\ge\aleph_{0}$. This relation can also be established by
exhibiting an appropriate invariant subspace for $T$, but in this
case the subspace can be chosen to be reducing for $T$, as well as
for the invariant subspace $\mathcal{H}'$.
\begin{thm}
\label{thm:Horn inequalities-infinite case}Assume that $T$ is an
operator of class $C_{0}$ on $\mathcal{H}$, $\mathcal{H}'$ is an
invariant subspace for $T$, $\mathcal{H}''=\mathcal{H}\ominus\mathcal{H}'$,
and
\[
T=\begin{bmatrix}T' & *\\
0 & T''
\end{bmatrix}
\]
is the matrix of $T$ corresponding to the orthogonal decomposition
$\mathcal{H}=\mathcal{H}'\oplus\mathcal{H}''$. Let
\[
\bigoplus_{1\le n<\aleph}S(\theta_{n}),\ \bigoplus_{1\le n<\aleph}S(\theta'_{n}),\ \bigoplus_{1\le n<\aleph}S(\theta_{n}^{\prime\prime})
\]
with be the Jordan models of $T,T'$, and $T''$, respectively. Fix
an ordinal $\beta\ge\aleph_{0}$. There exist separable reducing spaces
$\mathcal{M},\mathcal{M}',\mathcal{M}''$ for $T,T',T''$, respectively,
such that $\mathcal{M}=\mathcal{M}'\oplus\mathcal{M}''$ and the Jordan
models of $T|\mathcal{M}$, $T'|\mathcal{M}'$, and $T''|\mathcal{M}''$
are orthogonal sums of countably many copies of $S(\theta_{\beta}),S(\theta'_{\beta}),$
and $S(\theta_{\beta}^{\prime\prime})$, respectively. In particular,
$\theta_{\beta}|\theta'_{\beta}\theta_{\beta}^{\prime\prime}$.\end{thm}
\begin{proof}
We may assume without loss of generality that $\theta_{\beta}\not\equiv1$.
Consider quasidirect decompositions
\[
\mathcal{H}=\bigvee_{1\le n<\aleph}\mathcal{H}_{n},\quad\mathcal{H}'=\bigvee_{1\le n<\aleph}\mathcal{H}'_{n},\quad\mathcal{H}''=\bigvee_{1\le n<\aleph}\mathcal{H}_{n}^{\prime\prime},
\]
 into invariant spaces for $T,T',T'',$ respectively, such that $T|\mathcal{H}_{n}\sim S(\theta_{n})$,
$T'|\mathcal{H}'_{n}\sim S(\theta_{n}')$, and $T''|\mathcal{H}_{n}^{\prime\prime}\sim S(\theta_{n}^{\prime\prime})$
for all $n<\aleph$. Denote by $\beth$ the cardinality of $\beta$,
and let $\mathcal{K}$ be the smallest subspace of $\mathcal{H}$
which reduces both $T$ and the orthogonal projection $P_{\mathcal{H}'}$
and contains all the spaces $\mathcal{H}_{n},\mathcal{H}'_{n},\mathcal{H}_{n}^{\prime\prime}$
for $n<\beth$. Since $\beth$ is transfinite and each $\mathcal{H}_{n}$
is separable, it follows that the space $\mathcal{K}$ has dimension
at most $\beth$. Moreover, the first $\aleph_{0}$ inner functions
in the Jordan model of $T|\mathcal{K}^{\perp}$ are equal to $\theta_{\beta}$,
and similar statements hold for $T'|\mathcal{H}'\cap\mathcal{K}^{\perp}$
and $T''|\mathcal{H}''\cap\mathcal{K}^{\perp}$. Indeed, the minimal
function of $T|\mathcal{K}^{\perp}$ divides $\theta_{\beta}$ because
$P_{\mathcal{K}^{\perp}}[\bigvee_{{\rm card}(n)\ge\beth}\mathcal{H}_{n}]$
is dense in $\mathcal{K}^{\perp}$. On the other hand, the multiplicity
of $T|[\varphi(T)\mathcal{K}^{\perp}]^{-}$ is at least $\beth$ if
$\varphi$ is an inner function which does not divide $\theta_{\beta}$.
Thus the Jordan model of $T|\mathcal{K}^{\perp}$ contains at least
$\beth$ summands equal to $S(\theta_{\beta})$. Choose invariant
subspaces $\mathcal{L},\mathcal{L}',\mathcal{L}''$ for $T|\mathcal{K}^{\perp}$,
$T'|\mathcal{H}'\cap\mathcal{K}^{\perp}$, $T''|\mathcal{H}''\cap\mathcal{K}^{\perp}$,
respectively, such that the Jordan models of the three restrictions
are orthogonal sums of countably many copies of $S(\theta_{\beta}),S(\theta'_{\beta}),$
and $S(\theta_{\beta}^{\prime\prime})$, respectively. The desired
space $\mathcal{M}$ is obtained as the smallest subspace containing
$\mathcal{L}\cup\mathcal{L}'\cup\mathcal{L}''$ which reduces both
$T$ and $P_{\mathcal{H}'}$. We then define $\mathcal{M}'=\mathcal{M}\cap\mathcal{H}'$
and $\mathcal{M}''=\mathcal{M}\cap\mathcal{H}''$. The space $\mathcal{M}$
is separable, and the Jordan model $\bigoplus_{1\le n<\aleph_{0}}S(\varphi_{n})$
satisfies $\varphi_{n}|\theta_{\beta}$ (because $\theta_{\beta}$
is the minimal function of $T|\mathcal{K}^{\perp}$ and $\mathcal{K}^{\perp}\supset\mathcal{M}$)
and $\theta_{\beta}|\varphi_{n}$ (because $\mathcal{M}\supset\mathcal{L})$.
Thus $T|\mathcal{M}$ has the desired Jordan model. Similar arguments
determine the Jordan models of $T'|\mathcal{M}'$ and $T''|\mathcal{M}''$.
\end{proof}

\section{Comments on the inverse problem\label{sec:inverse problem}}

Let $T$ be an operator of class $C_{0}$ on $\mathcal{H}$, let $\mathcal{H}'$
be an invariant subspace for $T$, set $\mathcal{H}''=\mathcal{H}\ominus\mathcal{H}'$,
and let
\[
T=\begin{bmatrix}T' & *\\
0 & T''
\end{bmatrix}
\]
be the matrix of $T$ corresponding to the orthogonal decomposition
$\mathcal{H}=\mathcal{H}'\oplus\mathcal{H}''$.  Assume further that
\begin{equation}
J=\bigoplus_{1\le n<\aleph}S(\theta_{n}),\quad J'=\bigoplus_{1\le n<\aleph}S(\theta'_{n}),\quad J''=\bigoplus_{1\le n<\aleph}S(\theta_{n}^{\prime\prime}),\label{eq:three friends}
\end{equation}
are the Jordan models of $T$, $T'$, and $T''$, respectively. We
have seen that the functions $\{\theta_{n},\theta'_{n},\theta_{n}^{\prime\prime}\}_{1\le n<\aleph}$
are subject to a collection of `inequalities' of the form $\prod_{i\in I}\theta_{i}|\prod_{j\in J}\theta'_{j}\prod_{k\in K}\theta_{k}^{\prime\prime}$
for some finite equipotent sets $I,J,K\subset\{1,2,\dots\}$, and
also that $\theta_{n}|\theta'_{n}\theta_{n}^{\prime\prime}$ if $n\ge\aleph_{0}$.
Other necessary conditions are that $\theta'_{n}|\theta_{n}$ and
$\theta_{n}^{\prime\prime}|\theta_{n}$ for $1\le n<\aleph$. A natural
question is whether these conditions on $J,J',J''$ are sufficient
for the existence of an operator $T$ of class $C_{0}$ and of an
invariant subspace $\mathcal{H}'$ for $T$ such that $T\sim J$,
$T|\mathcal{H}'\sim J'$, and $P_{\mathcal{H}^{\prime\perp}}T|\mathcal{H}^{\prime\perp}\sim J''$.
The answer is in the negative: if $\mu_{T}\le N<\infty$ we have $\theta_{n}\equiv1$
for $n>N$ and the models also satisfy the `determinant' condition
\[
\prod_{n=1}^{N}\theta_{n}\equiv\prod_{n=1}^{N}(\theta'_{n}\theta_{n}^{\prime\prime}).
\]
In this special case (with $\theta_{n}\equiv1$ if $n>N)$, the relations
outlined above are in fact sufficient for the existence of $T$ and
$\mathcal{H}'$ (see \cite{li-mull1,li-mull2} and, for the algebraic
case, \cite{klein-2} or \cite{macD}). An appropriate substitute
for the determinant condition has not been found in the case of infinite
multiplicity. We argue that, at least, the problem reduces to the
separable case.
\begin{prop}
Let $J,J',$ and $J''$ be Jordan operators given by \emph{(\ref{eq:three friends})}.
Assume that:
\begin{enumerate}
\item $\theta_{n}\equiv\theta_{1}$, $\theta'_{n}\equiv\theta'_{1}$, and
$\theta_{n}^{\prime\prime}\equiv\theta_{1}^{\prime\prime}$ for all
$n\le\aleph_{0}$,
\item $\theta_{n}'|\theta_{n}$ and $\theta_{n}^{\prime\prime}|\theta_{n}$
for all $n<\aleph$, and
\item $\theta_{n}|\theta'_{n}\theta_{n}^{\prime\prime}$ for all $n<\aleph$.
\end{enumerate}

Then there exists an operator $T$ of class $C_{0}$ and an invariant
subspace $\mathcal{H}'$ for $T$ such that $T\sim J$, $T|\mathcal{H}'\sim J'$,
and $P_{\mathcal{H}^{\prime\perp}}T|\mathcal{H}^{\prime\perp}\sim J''$.

\end{prop}
\begin{proof}
Define $T=J$ on $\mathcal{H}=\bigoplus_{1\le n<\aleph}S(\theta_{n})$
and $\mathcal{H}'=\bigoplus_{1\le n<\aleph}\mathcal{H}'_{n}$, where
for every ordinal number written as $n=m+k$ with $m$ a limit ordinal
and $k<\aleph_{0}$ we set
\[
\mathcal{H}'_{n}=\begin{cases}
(\theta_{n}/\theta'_{n})H^{2}\ominus\theta_{n}H^{2} & \text{if \ensuremath{k}is even,}\\
(\theta_{n}/\theta_{n}^{\prime\prime})H^{2}\ominus\theta_{n}H^{2} & \text{if \ensuremath{k}is odd.}
\end{cases}
\]
It is easy to verify that these objects satisfy the requirements of
the proposition.
\end{proof}
This proposition shows that there is no need for more elaborate conditions
on the functions $\theta_{n}$ for $n\ge\aleph_{0}$. It also reduces
the inverse problem, which remains open, to the separable case.


\begin{thebibliography}{10}
\bibitem{belk}P. Belkale, Local systems on $\mathbb{P}^{1}-S$ for
$S$ a finite set. \emph{Compositio Math.} \textbf{129} (2001), 67--86.

\bibitem{three-test-pro}H. Bercovici, Three test problems for quasisimilarity.
\emph{Canad. J. Math}. \textbf{39} (1987), 880--892.

\bibitem{C0-book}---------, \emph{Operator theory and arithmetic
in} $H^{\infty}$. American Mathematical Society, Providence, RI,
1988.

\bibitem{ber-tan}H. Bercovici and A. Tannenbaum, The invariant subspaces
of a uniform Jordan operator. \emph{J. Math. Anal. Appl.} \textbf{156}
(1991), 220--230.

\bibitem{BLT-on-LR1}H. Bercovici, W. S. Li, and T. Smotzer, A continuous
version of the Littlewood-Richardson rule and its application to invariant
subspaces. \emph{Adv. Math.} \textbf{134} (1998), 278--293.

\bibitem{BLT-on-LR2}---------, Continuous versions of the Littlewood-Richardson
rule, selfadjoint operators, and invariant subspaces. \emph{J. Operator
Theory} \textbf{54} (2005), 69--92.

\bibitem{extremal-structure}H. Bercovici and W. S. Li, Invariant
subspaces with extremal structure for operators of class $C_{0}$.
\emph{Operator theory, structured matrices, and dilations}, 115--123,
Theta Ser. Adv. Math., 7, Theta, Bucharest, 2007. 

\bibitem{sing-val}H. Bercovici, B. Collins, K. Dykema, and W. S.
Li, Characterization of singular numbers of products of operators
in matrix algebras and finite von Neumann algebras. arXiv:1306.6434.

\bibitem{bcdlt}H. Bercovici, B. Collins, K. Dykema, W. S. Li, and
D. Timotin, Intersections of Schubert varieties and eigenvalue inequalities
in an arbitrary finite factor. \emph{J. Funct. Anal}. \textbf{258}
(2010), 1579--1627.

\bibitem{BDL-Szeged}H. Bercovici, K. Dykema, and W. S. Li, The Horn
inequalities for submodules. \emph{Acta Sci. Math. (Szeged)} \textbf{79}
(2013), 17--30.

\bibitem{young-tab}W. Fulton, \emph{Young tableaux}. \emph{With applications
to representation theory and geometry}. Cambridge University Press,
Cambridge, 1997.

\bibitem{fulton-survey}---------, Eigenvalues, invariant factors,
highest weights, and Schubert calculus. \emph{Bull. Amer. Math. Soc.
(N.S.)} \textbf{37} (2000), 209--249.

\bibitem{klein-1}T. Klein, The multiplication of Schur-functions
and extensions of p-modules. \emph{J. London Math. Soc.} \textbf{43}
(1968), 280--284.

\bibitem{klein-2}---------, The Hall polynomial. \emph{J. Algebra}
\textbf{12} (1969), 61--78.

\bibitem{kly}A. A. Klyachko, Stable bundles, representation theory
and Hermitian operators. \emph{Selecta Math. (N.S.)} \textbf{4} (1998),
419--445.

\bibitem{KT}A. Knutson, Allen and T, Tao, The honeycomb model of
${\rm GL}_{n}(\mathbb{C})$ tensor products. I. Proof of the saturation
conjecture. \emph{J. Amer. Math. Soc.} \textbf{12} (1999), 1055--1090.

\bibitem{KTW}A. Knutson, T. Tao, and C. Woodward, The honeycomb model
of ${\rm GL}_{n}(\mathbb{C})$ tensor products. II. Puzzles determine
facets of the Littlewood-Richardson cone. \emph{J. Amer. Math. Soc.}
\textbf{17} (2004), 19--48.

\bibitem{li-mull1}W. S. Li and V. M\"uller, Littlewood-Richardson
sequences associated with $C_{0}$-operators. \emph{Acta Sci. Math.
(Szeged)} \textbf{64} (1998), 609--625.

\bibitem{li-mull2}---------, Invariant subspaces of nilpotent operators
and LR-sequences. \emph{Integral Equations Operator Theory} \textbf{34}
(1999), 197--226.

\bibitem{macD}I. G. Macdonald, \emph{Symmetric functions and Hall
polynomials. Second edition. With contributions by A. Zelevinsky}.
Oxford Science Publications. The Clarendon Press, Oxford University
Press, New York, 1995.

\bibitem{moore-nor}B. Moore, III and E. A. Nordgren, On quasi-equivalence
and quasi-similarity. \emph{Acta Sci. Math. (Szeged)} \textbf{34}
(1973), 311--316.

\bibitem{nordg}E. A. Nordgren, On quasi-equivalence of matrices over
$H^{\infty}$. \emph{Acta Sci. Math. (Szeged)} \textbf{34} (1973),
301--310.

\bibitem{SN}B. Sz.-Nagy, C. Foias, H. Bercovici, and L. K\'erchy,
\emph{Harmonic analysis of operators on Hilbert space. Second edition.}
Springer, New York, 2010.

\bibitem{thompson}R. C. Thompson, Divisibility relations satisfied
by the invariant factors of a matrix product, \emph{The Gohberg anniversary
collection, Vol. I} Birkh\"auser, Basel, 1989, pp. 471--491. \end{thebibliography}
\end{document}